\documentclass[12pt, reqno, letterpaper]{amsart}
\usepackage[margin=1in]{geometry}

\usepackage{lmodern}
\usepackage[T1]{fontenc}
\usepackage{amssymb, amsmath, amsfonts, amsthm, mathrsfs, tikz-cd, hyperref}
\usepackage{cite}
\usepackage{dutchcal, bm}
\usepackage{graphicx}
\allowdisplaybreaks[1] 

\usepackage{fancyhdr} 
\usepackage{color}
\definecolor{linkcol}{rgb}{0,0,0.6}
\definecolor{hrefcol}{rgb}{0,0.5,0.5}
\definecolor{citecol}{rgb}{0.6,0,0.6}
\definecolor{amber}{rgb}{0.5,0.4,0.0}
\hypersetup{colorlinks=true,
            linkcolor=linkcol,
            filecolor=magenta,      
            urlcolor=hrefcol,
            citecolor=citecol}
\tikzcdset{arrow style=tikz, diagrams={>=stealth}}

\makeatletter
\def\l@subsection{\@tocline{2}{0pt}{4pc}{5pc}{}}
\makeatother



\numberwithin{equation}{section}
\newcommand\numeq{\stepcounter{equation}\tag*{{\sc(\arabic{section}.\arabic{equation})}}}

\makeatletter 
\renewcommand\th@plain{\slshape}
\makeatother

\newtheoremstyle{result}
  {1.7ex}
  {1.7ex}
  {\slshape}
  {0pt}
  {\scshape}
  {.}
  { }
  {\thmname{#1}\thmnumber{ #2}\thmnote{\if&#3& #3\else { }$-$ #3\fi}}

\newtheoremstyle{unique}
  {1.7ex}
  {1.7ex}
  {}
  {0pt}
  {\scshape}
  {.}
  { }
  {\thmnote{#3}}

\newtheoremstyle{other}
  {1.7ex}
  {1.7ex}
  {}
  {0pt}
  {\scshape}
  {.}
  { }
  {\thmname{#1}\thmnumber{ #2}\thmnote{\if&#3& #3\else { }$-$ #3\fi}}

\theoremstyle{result}
\newtheorem{THM}{Theorem}[section]
\newtheorem{PROP}[THM]{Proposition}

\newtheorem{COR}[THM]{Corollary}

\theoremstyle{unique}
\newtheorem{UNI}{}

\theoremstyle{other}
\newtheorem{DEF}[THM]{Definition}

\newtheorem{NOT}[THM]{Notation}

\newtheorem{RMK}[THM]{Remark}
\newtheorem{RMK!}[THM]{Important Remark}

\newtheorem{NB}[THM]{Nota bene}
\newtheorem{EXP}[THM]{Example}
\newtheorem{EXS}[THM]{Examples}

\def\bb{\mathbb}
\def\fk{\mathfrak}
\def\0{\varnothing}
\def\scr{\mathcal}
\def\<{\langle}
\def\>{\rangle}

\newcommand{\keywd}[1]{{\fontseries{b}\selectfont{\boldmath#1}}}

\newcommand{\nocontentsline}[3]{}
\let\origcontentsline\addcontentsline
\newcommand\stoptoc{\let\addcontentsline\nocontentsline}
\newcommand\resumetoc{\let\addcontentsline\origcontentsline}

\title{Floer homotopy theory for monotone Lagrangians}
\author{Ciprian M. Bonciocat}
\address{Department of Mathematics, Stanford University, 450 Jane Stanford Way, Building 380,
Stanford, CA 94305-2125, USA.}
\email{ciprianb@stanford.edu}

\begin{document}


\begin{abstract}
    We circumvent one of the roadblocks in associating Floer homotopy types to monotone Lagrangians, namely the curvature phenomena occurring in high dimensions. Given $N \ge 3$ and $R$ a connective $\mathbb E_1$-ring spectrum, there is a notion of an $N$-truncated, $R$-oriented flow category, to which we associate a module prospectrum over the Postnikov truncation $\tau_{\le N - 3}R$. This endows ordinary Floer cohomology with an action of the Steenrod algebra over $\tau_{\le N-3}R$, and also induces certain generalized cohomology theories. We give sufficient conditions for a closed embedded monotone Lagrangian to admit such well-defined invariants for $N = N_\mu$ the minimal Maslov number, and $R = MU$ complex bordism. Finally, we formulate Oh-Pozniak type spectral sequences for these invariants, and show that in the case of $\mathbb{RP}^n \subset \mathbb{CP}^n$ they provide further restrictions on the topology of clean intersections with a Hamiltonian isotopy, not detected by ordinary Floer (co)homology.
\end{abstract}
\maketitle

\setcounter{tocdepth}{2}
\tableofcontents


\newpage
\section*{Introduction}\label{sec:intro}

  Given any classical Floer homology theory, a Floer spectrum or \emph{Floer homotopy type} is a lift thereof to the stable homotopy category of (pro-)modules over some ring spectrum $R$. Ideally, $R$ would be the sphere spectrum, but in many situations this is (at least a priori) not the case. One main source of such stable homotopy upgrades is constituted roughly by the following additional input data:
  \begin{enumerate}
    \item A \emph{flow category}, i.e.~ coherent data of all the higher-dimensional compactified moduli spaces associated to the Floer equation.
    \item Suitably coherent \emph{$R$-orientations} associated to them.
  \end{enumerate}
  For foundational literature on Floer homotopy theory, see Manolescu's finite-dimensional approximation method \cite{ManSWF}, Cohen-Jones-Segal's construction \cite{CJS}, Cohen's subsequent generalization \cite{Cohen}, Large \cite{Large}, Abouzaid-Blumberg \cite{AB}, or various re-interpretations of the same methods e.g. Abouzaid-Blumberg \cite{ABMorava}, Cote-Kartal \cite{CoteKartal}, Porcelli-Smith \cite{PoSm1, PoSm2}. Throughout this paper, we closely follow and extend the viewpoint developed in our previous work \cite{me}. Applications of Floer homotopy theory are numerous, so we are obliged to only mention a few: Manolescu's Seiberg-Witten theory \cite{ManSWF}, Lipshitz-Sarkar's Khovanov theory \cite{LS}, Abouzaid-Blumberg's Arnold conjecture over $\bb F_p$ \cite{ABMorava}, Cohen's Viterbo isomorphism \cite{CohenCotangent}, equivariant structures of Cote-Kartal \cite{CoteKartal} and Rezchikov \cite{Rez}, and Blakey's lower bounds on degenerate Lagrangian intersections \cite{Blakey}.

  The standard methods of constructing a Floer homotopy type, e.g. \cite{CJS,AB}, really use the full power of part {\sc i} of the data, namely the existence of sufficiently nice compactified moduli spaces of \emph{all dimensions}. This already makes the theory difficult or even impossible to use in the presence of certain types of bubbling, e.g. non-exact/aspherical Hamiltonian and Lagrangian theories, or Instanton Floer theory; for a slightly more detailed discussion, see {\sc Exs.~\ref{exs:mot}} below. In the present paper, we provide a way of overcoming this difficulty, at the cost of downgrading the ring spectrum $R$ to a suitable Postnikov truncation of it (which is analogous to tensoring a module over a commutative ring $A$ by a quotient $A/\fk a$).
  
  We note that an earlier workaround discovered by Porcelli-Smith \cite{PoSm1} consists in using manifolds with Baas-Sullivan singularities (whose bordism theory at least morally should be the representing functor of the Postnikov truncation), and produce a generalized bordism (co)homology theory resembling a truncated $\pi_*^{\rm st}$ without building an actual Floer spectrum. However, it is not clear how to recover ordinary Floer cohomology and Steenrod operations from this bordism theory.

  Our two main theoretical results are stated below, in a compressed form. For more detailed statements, and where they are located throughout the paper, see the references cited therein, or the ``Outline'' below.

  \begin{UNI}[Theorem A]\sl
    An $N$-truncated $R$-oriented flow category $\scr F$ {\sc(Def.~\ref{def:truncfl})} induces a well-defined module pro-spectrum over the Postnikov truncation $\tau_{\le N - 3}R$ {\sc(Def.~\ref{def:post})}. This suffices to induce an action of the Steenrod algebra on the ordinary Floer cohomology of $\scr F$ with any coefficients $A$:
    \[ \scr A^*_A(\tau_{\le N - 3} R) \quad \circlearrowleft \quad {\rm HF}^*(\scr F; A),\]
    cf.~{\sc Defs.~\ref{def:hf},~\ref{def:steen}} and {\sc Thm.~\ref{thm:act}}. It also induces generalized (co)homology theories associated to $\scr F$, under certain finiteness assumptions, cf.~{\sc Def.~\ref{def:fin}} and {\sc Thm.~\ref{thm:gcoh}}.
  \end{UNI}

  \begin{UNI}[Theorem B]\sl
    A closed embedded monotone Lagrangian submanifold $L$ of a closed symplectic manifold $M$, satisfying the conditions written at the beginning of {\sc\S\ref{ssec:setup}}, and equipped with a $U$-brane {\sc(Def.~\ref{def:ubr}, Ex.~\ref{exp:sourceu})}, induces an action of the Steenrod algebra 
    \[ \scr A^*_A(\tau_{\le N_\mu - 3} MU) \quad \circlearrowleft \quad {\rm HF}^*(M,L;A),\]
    where $N_\mu$ is the minimal Maslov number. The action is independent of the auxiliary choices of Hamiltonian and almost-complex structure {\sc(Thm.~\ref{thm:welldef})}. Likewise, there are well-defined generalized Floer (co)homology theories, under suitable finiteness assumptions {\sc(Thm.\,\sc \ref{thm:welldefZ})}.

    The standard methods due to Oh, Pozniak {\sc(Thm.~\ref{thm:ohpoz})} and Albers, Piunikhin-Salamon-Schwarz {\sc(Thm.~\ref{thm:albpss})} for computing Lagrangian Floer cohomology from clean intersections also extend to modules over the Steenrod algebra, or to generalized (co)homology theories.
  \end{UNI}

  We remark that the existence of a $U$-brane actually forces $N_\mu$ to be even {\sc(Rmk.~\ref{rmk:constrbr})}, so in fact $\tau_{\le N_\mu - 3} MU = \tau_{\le N_\mu - 4} MU$ since complex bordism is trivial in odd degrees. At present, due to a lack of development in the area of pro-spectra, we cannot vouch for the independence of the pro-spectrum with respect to the Hamiltonian and almost-complex structure, partly due to the ambiguity of what such a notion of equivalence of pro-spectra really means, cf.~ {\sc Rmk.~\ref{rmk:prosp}} for a proposed definition, and the state of the literature. Before we introduce our application, we mention a certain class of Steenrod operations that lift to Postnikov truncations of complex bordism, that will be used in the application.

  \begin{UNI}[Fact C] (cf.~ \cite{Yag,Milnor}) \sl
    At any prime $p$, there are cohomology operations $Q_i \in \scr A_{\bb F_p}^{2p^{i} - 1}(\tau_{\le r} MU)$ for all $2p^i - 2 \le r$ (cf.~ {\sc Ex.~\ref{exp:miln}}, {\sc Prop.~\ref{prop:qimu}}), satisfying the following properties:
    \begin{enumerate}
      \item Their images in the ordinary Steenrod algebra $\scr A^*_{\bb F_p}(\bb S)$ recover the Milnor primitives, which satisfy the recursive definitions
      \[ Q_0 = b, \qquad Q_{i+1} = \scr P^{p^i}Q_i - Q_i \scr P^{p^i}, \]
      where $b$ is the Bockstein homomorphism, and $\scr P^n = {\rm Sq}^{2n}$ when $p = 2$.
      \item Leibniz rule: $Q_i(x\cdot y) = Q_i x \cdot y + (-1)^{|x|} x \cdot Q_i y$.
      \item Power rule: $Q_i x = x^{2^{i+1}}$ for all $x \in H^1(X; \bb F_2)$, where $X$ is an actual space; similarly $Q_i x = (bx)^{p^i}$ for all $x \in H^1(X; \bb F_p)$ at odd primes $p$.
    \end{enumerate}
    In particular, given $L \subset M$ as in {\sc Thm.\,B}, these operations exist on ${\rm HF}^*(M,L;\bb F_p)$ for all $i \ge 0$ satisfying $2p^i \le N_\mu - 1$.
  \end{UNI}
  
  As a sample application of this theory, we study clean intersections of $\bb{RP}^n$ with a Hamiltonian isotopy inside $\bb{CP}^n$. It already follows purely from Floer homology (cf.~ {\sc Thm.~\ref{thm:conn}}) that this clean intersection cannot be connected, unless it is the whole of $\bb{RP}^n$. The next simplest case to consider is when the intersection is a disjoint union of a point $\{p\}$ and another connected component $C$. Again, it follows purely from Floer homology (cf.~ {\sc Thm.~\ref{thm:pt+connI}}) that $C$ must have the same dimension and singular $\bb F_2$-homology as $\bb{RP}^{n-1}$. Our main application extracts a bit of extra information about $C$, using the Steenrod action of the Milnor primitives:

  \begin{UNI}[Theorem D]\sl
    For odd $n \ge 3$, 
    \begin{enumerate}
      \item $\bb{RP}^n \subset \bb{CP}^n$ admits a diagonal-respecting $U$-brane {\sc(Def.~\ref{def:ubr})}, and induces a well-defined Steenrod algebra action over $\tau_{\le n-3}MU$, as in {\sc Thm.\,B.}
      \item If $\Phi_t$ is a non-autonomous Hamiltonian flow such that $\bb{RP}^n \cap \Phi_1(\bb{RP}^n)$ is a clean intersection made up of a disjoint point $\{p\}$, and a connected manifold $C$, then the cohomology of $C$ satisfies certain relations involving the mod-2 $Q_i$ operations of {\sc Fact\,C}, which are also satisfied by $\bb{RP}^{n-1}$ (cf.~ {\sc Thm.~\ref{thm:pt+connII}, Cor.~\ref{cor:pt+conn}} for precise statements).
      \item Also, certain mod-2 characteristic classes $q_i$ {\sc(Def.~\ref{def:qi})} of the tangent bundle of $C$ are guaranteed to be nonzero (cf.~{\sc Thm.~\ref{thm:pt+connIII}} for the precise formulation.)
    \end{enumerate}
  \end{UNI}

  We conjecture that $C$ must be at least homotopy-equivalent to $\bb{RP}^{n-1}$, as this is the only example we could find of such a clean intersection, cf.~{\sc Rmk.~\ref{rmk:exofc}}. We end the paper with {\sc Rmk.~\ref{rmk:next}}, in which we discuss further potential extensions and consequences of our work.

  \stoptoc\subsection*{Outline} Here is a more detailed list of our definitions and theorems.
  \begin{enumerate}
    \item In {\sc Def.~\ref{def:truncfl}}, we introduce the notion of an $N$-truncated flow category, containing moduli spaces only up to dimension $N-2$, and discuss $G$-structures {\sc(Def.~\ref{def:g-flow})} and more generally $R$-orientations {\sc(Def.~\ref{def:r-or-flow})} on flow categories.
    \item In {\sc Def.~\ref{def:steen}}, we define the Steenrod algebra $\scr A^*_M(R)$ over an $\bb E_1$-ring spectrum $R$ with $M$-coefficients. Then, we state {\sc Thm.\,A} in {\sc Thms.~\ref{thm:act},~\ref{thm:gcoh}}, and prove them in {\sc\S\ref{ssec:act}}, by constructing a pro-spectrum \`a la Cohen-Jones-Segal. The Milnor $Q_i$ operations of {\sc Fact\,C} are also reviewed in this section, cf.~ {\sc Ex.~\ref{exp:miln}} and {\sc Prop.~\ref{prop:qimu}}.
    \item In {\sc\S\ref{sec:apps}}, we discuss the theory of monotone Lagrangians, under certain topological assumptions, cf.~ beginning of {\sc\S\ref{ssec:setup}}. We construct an $N_\mu$-truncated flow category $\scr F(M,L)$ in {\sc Thm.\,\ref{thm:lagflow}} using Large's smooth gluing \cite{Large}, and also set up an $\bb E_1$-coherent index theorem for abstract holomorphic strips, cf.~ {\sc Defs.\,\ref{def:aps}, \ref{def:ind}} and {\sc Thm.\,\ref{thm:ind}}. 
    \item We then define the notion of a $U$-brane {(\sc Def.~\ref{def:ubr})} for the Lagrangian $L \subset M$, which by our index theorem endows $\scr F(M,L)$ with a $U$-structure, inducing a Steenrod algebra action on classical Floer cohomology ${\rm HF}^*(M,L)$ over the ring spectrum $\tau_{\le N_\mu - 3}MU$ {\sc(Thm.~\ref{thm:welldef})}, cf.~ {\sc Thm.\,B} above, as well as generalized cohomology theories with finite homotopy-support ({\sc Def.~\ref{def:fin}, Thm.~\ref{thm:welldefZ}}).
    \item We further upgrade the classical Albers-PSS isomorphism {(\sc Thm.~\ref{thm:albpss})} and the spectral sequences due to Oh and Pozniak {(\sc Thm.~\ref{thm:ohpoz})}, to graded modules over the Steenrod algebra, and/or in generalized (co)homology theories, which are useful for computing Floer cohomology from clean intersections.
    \item We end with applications of our theory to $\bb{RP}^n \subset \bb{CP}^n$, showing how the Steenrod algebra structure provides further restrictions on the topology of possible clean intersections of $\bb{RP}^n$ with a Hamiltonian isotopy, cf.~ the preliminary results {\sc Thms.~\ref{thm:conn}, \ref{thm:pt+connI}} using only classical Floer cohomology, and then {\sc Thms.~\ref{thm:pt+connII},~\ref{thm:pt+connIII}, Cor.~\ref{cor:pt+conn}} using the Milnor $Q_i$ operations to strengthen the conclusion of {\sc Thm.~\ref{thm:pt+connI}}. Finally, in {\sc Rmk.~\ref{rmk:next}}, we discuss further ramifications of these ideas.
  \end{enumerate}

  \subsection*{Acknowledgments} I would like to thank my advisor Ciprian Manolescu, supported by the Simons Investigator Grant, for directing me and providing many useful suggestions throughout the writing process, as well as Mohammed Abouzaid, Kenneth Blakey, Andrew Blumberg, Sebastian Haney, Eleny Ionel, Bari\c s Kartal, Noah Porcelli, Ivan Smith, and Hiro Lee Tanaka, for illuminating comments and discussions. I am also greatly indebted to my sponsors for the William R. and Sara Hart Kimball Fellowship offered by Stanford University, and to the Simons Foundation for organizing many conferences that benefitted my research.

  \resumetoc
  \newpage
\section{General construction}\label{sec:constr}

  \subsection{Truncated flow categories}\label{ssec:truncflow} In this subsection, we review the notion of flow category \cite[\sc Def.\,2.9]{me}, suitably modified to the scope of our current work. Here are the notable differences:
  \begin{enumerate}
    \item We allow the indexing set $(I, <)$ to be potentially unbounded both above and below.
    \item We do not work in Morse-Bott generality here (i.e.~ the $\scr C_i$'s of \cite{me} will be points).
    \item Most importantly, we only require moduli spaces to exist up to a specified dimension, hence the adjective \emph{truncated}.
  \end{enumerate}

  \begin{DEF}[$I$-stratifications and $I$-corners]
    Given a set $I$, an \keywd{$I$-stratified} space is a topological space $X$, with a collection of closed subsets $\{\partial^i X \subset X\}_{i \in I}$. An $n$-dimensional smooth manifold with \keywd{$I$-corners} is a an $I$-stratified space with an atlas of charts modelled after Euclidean spaces with $I$-corners $\mathbf R^{n - |J|} \times \mathbf R_{\ge 0}^J$ for subsets $J \subset I$ of size at most $n$, with strata given by the faces of the corner. For more details, see \cite[\sc Def.\,2.3]{me}.
  \end{DEF}

  \begin{DEF}[Truncated flow category]\label{def:truncfl} Cf.~ our \cite[\sc Def.~2.9]{me}, also Porcelli-Smith \cite[\sc Defs.~5.28,~5.32]{PoSm1}
    Given $3 \le N \le \infty$, we define an \keywd{$N$-truncated} flow category to be a collection of the following data:
    \begin{enumerate}
      \item A totally ordered set $(I,<)$, embeddable in $(\bb Z, <)$.
      \item A grading function $\mu : I \to \bb Z$, with the property that $\mu^{-1}(d)$ is finite for all $d \in \bb Z$.
      \item For every pair $i < j$ of indices in $I$ with $\mu(i) - \mu(j) < N$, a compact smooth $(\mu(i) - \mu(j) - 1)$-dimensional manifold with $\{u : i < u < j\}$-corners $\scr M_{i,j}$.
      \item\label{i:ident} Stratification-respecting identifications $\scr M_{i,u} \times \scr M_{u,j} \cong \partial^u \scr M_{i,j}$, which are associative in the most straightforward sense.
    \end{enumerate}
  \end{DEF}

  \begin{EXS}[Motivation for truncatedness]\label{exs:mot} The following cases illustrate the ubiquity of this notion:
    \begin{enumerate}
      \item \keywd{Hamiltonian Floer theory.} Given a monotone symplectic manifold $M$, its \emph{(twice) minimal Chern number} $N_{2c_1}$ is the positive generator of the image of the homomorphism $\int_{S^2} 2c_1(TM) : \pi_2(M) \to \bb Z$. For moduli spaces with $\mu(i) - \mu(j) < N_{2c_1}$, compactness up to broken trajectories can be ensured. For $\mu(i) - \mu(j) < N_{4c_1} = 2N_{2c_1}$, the same is true but up to potentially breaking off a simple bubble, which can be treated as part of the interior due to the $\mathbb C^\times$-parameter space required to glue. At least prima facie, one cannot go higher, because (derived) orbifold structures may become necessary for compactifying the moduli spaces. In theory, a method similar to the Cohen-Jones-Segal construction would be possible, after inverting certain primes, but \cite{AB} is probably a better framework for handling this.
      \item \keywd{Lagrangian Floer theory.} Given also monotone Lagrangians $L_0, L_1 \subset M$, their \emph{minimal Maslov number} $N_\mu$ is the minimum of the positive generators of the images of the two Maslov indices $\mu_{0,1} : \pi_2(M,L) \to \bb Z$. One always has $N_\mu \le N_{2c_1}$. For moduli spaces with $\mu(i) - \mu(j) < N_\mu$, compactness can again be ensured up to broken trajectories. As soon as we go beyond this limit, disk bubbles (even when they are simple) produce actual boundary, since the gluing parameter space is $(0,\epsilon)$. Treating them as ``interior'' would make the bordism theory trivial, so using \cite{AB}'s setup does not solve the problem. One could define curved versions of flow categories, but it becomes unclear how to assign homotopy types to them. If $L_0, L_1$ are Hamiltonian isotopic, there seems to be a way to produce cancelling cobordisms for disk bubbles appearing either on $L_0$ or $L_1$, which can be used to kill the extra boundary, cf.~{\sc Rmk.~\ref{rmk:impr}}. However, beyond $N_{2\mu} = 2N_{\mu}$, the same orbifold issue arises.
      \item \keywd{Instanton Floer theory.} Depending on the particular flavor, there are two sources of trouble: reducible representations, which break smoothness of the infinite-dimensional manifold, and bubbling (which typically introduces singularities of very high codimension). The former can at times either be avoided entirely, by imposing some ``admissibility condition'' on the manifold. The latter only occurs when $\mu(i) - \mu(j) < N = 4h^\vee$, where $h^\vee$ is the dual Coxeter number of the Lie group $G$ (in particular, $N = 4$ for $SO(3)$-theory, and $N = 8$ for $SU(2)$-theory), cf.~ \cite{KrMr}.
    \end{enumerate}
    In the first two cases, smoothness is guaranteed by Large \cite{Large}, building on exponential decay estimates due to Fukaya-Oh-Ohta-Ono \cite{FOOOexp}. In the latter, the existence of smooth corner structures remains open, although we expect similar exponential decay techniques to work. We should also mention that Instanton theory can be modelled (at least conjecturally) as a singular Lagrangian intersection problem, cf.~ \cite{DF,DFL}, hence potentially allowing one to still use \cite{Large} directly to deal with smoothness.
  \end{EXS}

  \begin{RMK}
    Anticipating the Cohen-Jones-Segal construction of {\sc\S\ref{ssec:steenact}}, the moduli spaces of dimension up to $N-3$ should be thought of as data, whereas the existence of the top ($N-2$)-dimensional ones should be thought of as property. A good analogy is classical Floer homology, in which case the 0-dimensional moduli spaces provide the actual data required to build the chain complex, whereas the 1-dimensional moduli spaces are only required to exist. Because the top-dimensional moduli spaces in {\sc\S\ref{sec:apps}} will be canonically defined as well, we have decided to leave {\sc Def.~\ref{def:truncfl}} as it is.
  \end{RMK}

  In order to set the stage for the Pontrjagyn-Thom construction needed later, we now introduce the notion of embeddings:

  \begin{DEF}[Embeddings]\label{def:emb}
    Given $\scr F$ an $N$-truncated flow category, an \keywd{embedding} of $\scr F$, cf.~ \cite[\sc Def.\,2.14]{me}, is the following data:
    \begin{enumerate}
      \item\label{i:br} For all $i < j$ in $I$, Euclidean spaces $\mathbf E_{i,j}$ with $\{u : i<u<j\}$-corners (i.e.~ of the form $V \times \mathbf R_{\ge 0}^{\{i < u < j\}}$, where $V$ is a vector space), together with stratification-respecting identifications $\mathbf E_{i,u} \times \mathbf R_{\ge 0}^{\{u\}} \times \mathbf E_{u,j} \cong \mathbf E_{i,j}$, that are associative in the most straightforward manner.
      \item For all $i < j$ in $I$, smooth stratification-preserving embeddings $\iota_{i,j} : \scr M_{i,j} \hookrightarrow \mathbf E_{i,j}$ meeting the boundary strata in the Euclidean space transversely, and satisfying the boundary compatibility condition $\iota_{i,j}|_{\partial^u} = \iota_{i,u} \times \{0\} \times \iota_{u,j}$.
    \end{enumerate}
    A \keywd{stabilization} of the embedding consists in adding vector spaces consistently to all the $\mathbf E_{i,j}$, and extending the embedding by the zero map on the newly added summands.
  \end{DEF}

  \begin{PROP}[Existence and homotopy-uniqueness]\label{prop:exun}
    Such an embedding always exists. Moreover, given any continuous (resp. smooth) $\partial D^k$-parameter family of such embeddings, one can extend it to a continuous (resp. smooth) $D^k$-parameter family of embeddings after suitable stabilization.
  \end{PROP}

  \begin{proof}
    In the case that $I$ is finite, this follows easily by applying the relative version of Whitney's embedding theorem over and over again, cf.~ \cite[\sc Prop.\,5.2]{me}. One only needs to ensure that the $\mathbf E_{i,j}$ are all big enough to being with. When $I$ is infinite, one can induct on larger and larger finite sub-segments of $I$, eventually exhausting $I$. Likewise, the parametric version also follows by the parametric relative version of Whitney's embedding theorem instead.
  \end{proof}

  \subsection{\texorpdfstring{$G$}{G}-structures on flow categories}\label{ssec:gstruc} We now discuss a certain extra kind of data, called $G$-structures, which often accompany flow categories arising from various different Floer theories. We begin with the following consequence of {\sc Prop.~\ref{prop:exun}}:

  \begin{RMK}
    From a homotopy-theoretic standpoint, {\sc Prop.~\ref{prop:exun}} implies that an embedding can always be chosen ``for free.'' It also means that the normal bundles $N_{i,j}$, viewed as rank-0 virtual bundles (i.e.~ maps $\scr M_{i,j} \to BO$) are well-defined up to contractible choice. Note that $N_{i,j}|_{\partial^u} \cong N_{i,u} \boxplus N_{u,j}$ as vector bundles, and so when passing to virtual bundles we get an $\bb E_1$-map $\scr M_{*,*} \overset{N_{*,*}} \longrightarrow BO$, with respect to the inclusions {\sc \ref{def:truncfl}-\ref{i:ident}} and the $\bb E_1$-multiplication on $BO$. Since $\scr M_{i,j}$ is only well-defined when $\mu(i) - \mu(j) < N$, it should implicitly be understood that these maps, together with the $\bb E_1$-operad action, are only defined in the appropriate range.
  \end{RMK}

  \begin{DEF}
    Let $G$ denote some $\bb E_2$-group, with an $\bb E_2$-map $G \to O$, inducing an $\bb E_1$-map $BG \to BO$. We define a \keywd{stable normal $G$-structure} on $\scr M_{*,*}$ to be a homotopy-lift
    \[\numeq\label{eq:glift}\begin{tikzcd}
      & BG \dar \\
      \scr M_{*,*} \rar["N_{*,*}"]\ar[ru, dotted] & BO
    \end{tikzcd}\]
    of $\bb E_1$-maps. (Examples include $G = SO, U$ or $1$.)
  \end{DEF}

  \begin{RMK}
    The bundles $\tilde T\scr M_{*,*} := \mathbf R \oplus T\scr M_{*,*}$ also induce an $\bb E_1$-map $\scr M_{*,*} \overset{\tilde T_{*,*}}\longrightarrow BO$. In virtue of the $\bb E_\infty$-structure on $BO$, one can add $\tilde T_{*,*}$ and $N_{*,*}$, to obtain a new $\bb E_1$-map $\tilde T_{*,*} + N_{*,*}$. This is canonically identified with the zero map, since the tangent and the normal bundles add up to the trivial bundle, i.e.~ $\mathbf R \oplus T\mathbf E_{i,j}$. Hence, $N_{*,*} = -\tilde T_{*,*}$.
  \end{RMK}

  \begin{DEF}
    A \keywd{stable tangential $G$-structure} on $\scr M_{*,*}$ is a homotopy lift similar to \ref{eq:glift}, but for $\tilde T_{*,*}$ instead.
  \end{DEF}

  \begin{COR}
    If $G \to O$ is in fact a map of $\bb E_\infty$-groups, then there is a canonical correspondence between stable tangential and normal $G$-structures, given by taking the inverse with respect to the $\bb E_\infty$-group law.
  \end{COR}

  \begin{DEF}\label{def:g-flow}
    If $G \to O$ is an $\bb E_\infty$-group map, then an \keywd{$N$-truncated $G$-flow category} is an $N$-truncated flow category with a stable $G$-structure in either of the two equivalent senses mentioned above.
  \end{DEF}

  \begin{RMK}\label{rmk:index}
    One may wonder how such a $G$-structure would arise in practice. More will be said in {\sc\S\ref{sec:apps}}, but for now we remark that according to Porcelli-Smith \cite{PoSm2}, in Lagrangian Floer theory (and very likely in other Floer theories as well), $\tilde T_{i,j}$ can be identified with the index bundle of the linearized Fredholm operator at the given solutions. Since $\scr{Fred}_{\bb R}$, the space of Fredholm operators on a separable infinite-dimensional real Hilbert space, is a classifying space for $KO$-theory, and this can be used to endow Lagrangian flow categories with interesting $G$-structures via index theory.
  \end{RMK}

  \subsection{Orientations in ring spectra, and Pontrjagyn-Thom maps}\label{ssec:orpt} Let $R$ be an $\bb E_1$-ring spectrum. Since $\bb E_1$-algebras can always be made strictly unital and associative, one can always model such a ring spectrum as an orthogonal ring spectrum \cite{Sch}, which we will have to do in {\sc\S\ref{ssec:act}}.

  \begin{EXP}
    If $\rho : G \to O$ is an $\bb E_2$-group map, then $B\rho : BG \to BO$ is an $\mathbb E_1$-group map, and hence the \keywd{$G$-Thom spectrum} $MG := {\rm Th}((B\rho)^*(\gamma))$ of the pull-back of the universal bundle $\gamma$ on $BO$ via $B\rho$, is an $\mathbb E_1$-ring spectrum. The original reference is Mahowald \cite{Mah}, cf.~ also May \cite[\sc\S IV.3]{May06}. In modern language, this follows from the fact that ${\rm Th}(-)$ is a symmetric monoidal $\infty$-functor with respect to external sums of virtual bundles and smash products of spectra.
  \end{EXP}

  The goal of this subsection is to produce, given enough data on the flow category, Pontrjagyn-Thom maps in $R$-modules, which will serve as building blocks for the Cohen-Jones-Segal construction later. If $R = MG$, then the data of a $G$-structure on the flow category should suffice to produce these $R$-maps, although it's convenient to work with more general $R$, as will be apparent in {\sc\S\ref{ssec:steenact}}, when we use Postnikov truncations.

  \begin{DEF}[$R$-oriented flow categories]\label{def:r-or-flow}
    Ler $\scr F$ be an $N$-truncated flow category. The Thom spectra ${\rm Th}(N_{*,*})$, where the virtual bundles $N_{i,j}$ are now remembered with their rank (i.e.~ as maps to $\bb Z \times BO$), form an $\bb E_1$-algebra. We also have an $\bb E_1$-algebra $R_{*,*}$ defined by $R_{i,j} := \Sigma^{{\rm rk\;}N_{i,j}} R$, and multiplication induced from the $\bb E_1$-structure on $R$. We define an \keywd{$R$-orientation} of $\scr F$ to be an $\bb E_1$-algebra map in spectra ${\rm Th}(N_{*,*}) \to R_{*,*}$, with the requirement that each individual map ${\rm Th}(N_{i,j}) \to R_{i,j}$ is an $R$-Thom class for $N_{i,j}$.
  \end{DEF}

  \begin{DEF}[Pontrjagyn-Thom collapse maps]\label{def:pt}
    Given an $R$-oriented $N$-truncated flow category $\scr F$, we define collapse maps
    \[ {\rm PT}_{i,j} : \mathbf E_{i,j}^+ \to \mathbf E_{i,j}/(\mathbf E_{i,j} - \scr M_{i,j}) \cong \bb {\rm Th}(N_{i,j}) \to R_{i,j}. \]
    Since the $\mathbf E_{i,j}^+$ form an $\bb E_1$-algebra with respect to the smash product (again, we only really consider this for $\mu(i) - \mu(j) < N$), via the inclusions 
    \[ \mathbf E_{i,u} \times \mathbf E_{u,j} = \mathbf E_{i,u} \times \{0\} \times \mathbf E_{u,j} \hookrightarrow \mathbf E_{i,j} \]
    (cf.~ {\sc Def.~\ref{def:emb}-\ref{i:br}}) we get an $\bb E_1$-map ${\rm PT}_{*,*} : \mathbf E_{*,*}^+ \to R_{*,*}$, the \keywd{Pontrjagyn-Thom} collapse map.
  \end{DEF}

  \begin{RMK}
    If $R = MG$, and $\scr F$ is endowed with a $G$-structure, then we automatically get $MG$-orientations, and hence $MG$-Pontrjagyn-Thom collapse maps. This is because the classifying map of $N_{i,j}$ factors through $BG$, leading to a Thomified map ${\rm Th}(N_{i,j}) \to \Sigma^{{\rm rk\;} N_{i,j}}MG$. We also get the same for any ring spectrum $R$ which is an algebra over $MG$.
  \end{RMK}

  \subsection{Steenrod action and generalized Floer (co)homology}\label{ssec:steenact} Let $R$ be now a \emph{connective} $\bb E_1$-ring spectrum, with $k := \pi_0 R$. We can associate (co)chain complexes to $N$-truncated $R$-oriented flow categories:

  \begin{DEF}[Floer (co)homology]\label{def:hf}
    Given a right (left) $k$-module $M$, and an $N$-truncated $R$-oriented flow category $\scr F$ (with $N \ge 3$), we define its associated \keywd{Floer (co)chain complexes} ${\rm CF}_*(\scr F; M)$ and ${\rm CF}^*(\scr F; M)$ by the following construction. We let ${\rm CF}_n(\scr F; k)$ be the free left $k$-module generated by $\mu^{-1}(n)$ (which was assumed to be finite, cf.~ {\sc Def.~\ref{def:truncfl}}), endowed with differential $d : {\rm CF}_n \to {\rm CF}_{n-1}$ given by
    \[ d(p) = \sum_{\mu(q) = \mu(p)-1} \#_R \scr M_{p,q} \cdot q,\]
    where $\#_R$ of a compact $R$-oriented 0-manifold $X$ is the integral of the $R$-Thom class over $X$. One has $d^2 = 0$ because $\bigsqcup_q \scr M_{p,q} \times \scr M_{q,r}$ for $\mu(r) = \mu(q) - 1 = \mu(p) - 2$ bounds the compact $R$-oriented 1-manifold $\scr M_{p,r}$. Then, one defines
    \[ {\rm CF}_*(\scr F; M) := M \otimes_k {\rm CF}_*(\scr F; k), \qquad {\rm CF}^*(\scr F; M) := {\rm Hom}_k({\rm CF}_*(\scr F; k), M), \]
    as usual. The Floer (co)homology is denoted by ${\rm HF}$ instead of ${\rm CF}$.
  \end{DEF}

  \begin{DEF}[Steenrod algebra]\label{def:steen}
    Since $R$ was assumed to be connective, there is an $\bb E_1$-ring map $R \to \tau_{\le 0}R = Hk$, so that a left $k$-module $M$ induces a left $\bb E_1$-module spectrum $HM$ over $R$. We define the \keywd{$M$-Steenrod algebra over $R$}, as
    \[ \scr A_M^*(R) := \pi_{-*}{\rm RHom}_R(HM, HM). \]
    This pedantic notation is used to clarify that we are taking homotopy-classes of maps $HM \to HM[n]$ in the derived category of $R$-modules. This is an algebra in the obvious way via composition, and it more generally acts on \keywd{$M$-coefficient cohomology over $R$}
    \[ H^*_R(X; M) := \pi_{-*}{\rm RHom}_R(X, HM) \]
    for any $R$-module $X$.
  \end{DEF}

  \begin{NB}[Terminology clash]
    We emphasize here that our definition of $H^*_R(X; M)$ given above does not agree with the cohomology $H^*(X; M) := [X, \Sigma^* HM]$ taken in the category of $\bb S$-modules. Indeed, the former parameterizes \emph{$R$-module} maps, so it takes the $R$-action on both $X$ and $HM$ into account.

    Also, another unfortunate confusion arises from the other reasonable notion of an $R$-Steenrod algebra, namely $R^*R$. We hope that the extra $M$ in the name of the ``$M$-Steenrod algebra over $R$'' will disambiguate this. The notion $R^*R$ will never be used in our paper.
  \end{NB}

  \begin{EXP}[Milnor primitives]\label{exp:miln}
    A great class of operations on the $H\bb F_p$-cohomology of any $MU$-module come from \keywd{Milnor primitives}. Indeed, recall that there are primitive elements $Q_i^{(p)}$ in the classical mod-$p$ Steenrod algebra, of degree $2p^i - 1$, defined recursively by
    \[ Q_0 := b, \qquad Q_{i+1} = \scr P^{p^i} Q_{i} - Q_{i} \scr P^{p^i}, \]
    where $b$ is the Bockstein homomorphism, and $\scr P^n$ is the Steenrod $n^{\rm th}$ power operation, with the convention that $\scr P^n = {\rm Sq}^{2n}$ at $p = 2$. We often omit the superscript $(p)$ if it is clear from the context. See \cite{Milnor} for the original reference. These satisfy the following elementary relations in the Steenrod algebra:
    \begin{enumerate}
      \item Exterior algebra: $Q_i^2 = 0$ and $Q_iQ_j + Q_j Q_i = 0$;
      \item Leibniz rule: $Q_i(x \otimes y) = Q_ix \otimes y + (-1)^{|x|} x \otimes Q_i y$;
      \item Power rule: $Q^{(2)}_i(x) = x^{2^{i+1}}$ for all classes $x \in H^1(X;\bb F_2)$ on an actual space $X$, and similarly $Q^{(p)}_i(x) = (bx)^{p^i}$ for all $x \in H^1(X;\bb F_p)$ at odd primes $p$.
    \end{enumerate}
    The relevance of these Milnor primitives is that they can be upgraded to Steenrod operations over $MU$, and even over $\tau_{\le r} MU$ if $i$ is sufficiently small, which we now review:
  \end{EXP}

  \begin{DEF}[Postnikov truncations]\label{def:post}
    We use the notation $\tau_{\le r} X$ to denote the $r$-th \keywd{Postnikov truncation} of a spectrum $X$, i.e.~ the spectrum characterized by the existence of a map $X \to \tau_{\le r} X$, which induces an isomorphism on $\pi_{* \le r}$, and such that $\pi_k (\tau_{\le r} X) = 0$ for $k > r$. It turns out this can also be done in the category of $\mathbb E_1$ (or $\mathbb E_\infty$) connective ring spectra, and that the forgetful functor to spectra commutes with $\tau_{\le r}$. This can be seen either by using the small object argument in ring spectra, or by using the fact that $\tau_{\le r}$ is lax (symmetric) monoidal with respect to smash products, as long as one restricts to connective ring spectra. For some references, see \cite[\sc\S8]{Laz}, \cite{Dug06}, \cite[\sc\S8]{Bas}, \cite[\sc\S7.4]{HA}.
  \end{DEF}

  \begin{PROP}[Milnor $Q_i$ over $\tau_{\le r}MU$]\label{prop:qimu}
    The operations $Q_i \in \scr A^*_{\bb F_p}(\bb S)$ admit canonical lifts to $Q_i \in \scr A^*_{\bb F_p}(MU)$. Moreover, there exist lifts $Q_i \in \scr A^*_{\bb F_p}(\tau_{\le r} MU)$, at least when $r \ge 2p^i - 2$.
  \end{PROP}

  \begin{proof}
    Recall that (connective) \emph{Morava k-theories} $k^{(p)}(n)$ are algebras over $MU$ satisfying
    \[ \pi_*k^{(p)}(n) = \bb F_p[v_n], \qquad |v_n| = 2p^n - 2. \]
    The first k-invariant of this spectrum, i.e.~ the connecting homomorphism $\partial$ associated to the exact triangle
    \[ H\bb F_p[2p^n - 2] \to \tau_{\le 2p^n - 2} k^{(p)}(n) \to H\bb F_p \overset\partial\to H\bb F_p[2p^n - 1], \] 
    is equal to $Q_n$, up to a potential unit in $\bb F_p^\times$. For a reference to this, see the original \cite[\sc Lem.\,2.1]{Yag} for general $p$, and \cite{KL24} for $p = 2$. Now, this construction can be done entirely in the category of $MU$-modules, which proves that $Q_i$ in fact lift to $MU$-operations.
    
    If $i$ is sufficiently small, these $Q_i$ in fact lift all the way to the Steenrod algebra over the truncation $\tau_{\le r} MU$. Indeed, as long as $r \ge 2p^i - 2$, the truncation $\tau_{\le r} k^{(p)}(i)$ is a module over $\tau_{\le r} MU$, and it still has the same non-trivial k-invariant, allowing one to repeat the process in the category of $\tau_{\le r} MU$-modules.
  \end{proof}

  \begin{RMK}
    We conjecture and are quite confident that the $Q_i$'s are the only ordinary cohomology operations that lift to (indecomposable) cohomology operations over Postnikov truncations of $MU$, but we do not treat this subject here, as it would be too great of a digression. There are also plenty of non-zero cohomology operations over $\tau_{\le r} MU$ that die even in the next $\tau_{\le r + 2} MU$, but this is a great subject for another paper.
  \end{RMK}

  We now finally state our main results, which we prove in the following subsection.

  \begin{THM}[Steenrod action on Floer (co)homology]\label{thm:act}
    Let $R$ be a connective $\bb E_1$-ring spectrum, with $k := \pi_0 R$, and $M$ be a left module over $k$. An $R$-oriented $N$-truncated flow category $\scr F$ naturally gives rise to a graded action
    \[\scr A^*_{M}(\tau_{\le N-3} R) \;\circlearrowleft\; {\rm HF}^*(\scr F; M),\]
    where $\tau_{\le N-3}R$ is the Postnikov truncation of $R$, cf.~ {\sc(Def.~\ref{def:post})} above.
  \end{THM}

  \begin{RMK}
    When $N = 3$, an $N$-truncated flow category does not carry any more information than is already needed to construct the Floer (co)chain complex. This is confirmed by our result, since $\tau_{\le N-3} R = Hk$, and so the Steenrod algebra is just ${\rm Ext}^*_k(M,M)$, which acts in the obvious, purely chain-level fashion, on ${\rm HF}^*(\scr F; M)$. Thus, $N \ge 4$ is necessary in order to get interesting results that go beyond mere homological algebra.
  \end{RMK}

  Even though we will not use the following notions in our application section {\sc\S\ref{sec:apps}}, nevertheless they are still important consequences of our general theory.

  \begin{DEF}[Finiteness property]\label{def:fin}
    If $N < \infty$, a $\tau_{\le N-3} R$-module $Z$ is said to have \keywd{finite homotopy-support} if $\pi_n Z$ is zero except for finitely many $n$.
  \end{DEF}

  \begin{THM}[Generalized Floer (co)homology theories]\label{thm:gcoh}
    Let $R$ be a connective $\bb E_1$-ring spectrum, and $Z$ a right (left) module over $\tau_{\le N-3}R$. An $R$-oriented $N$-truncated flow category $\scr F$ admits well-defined \keywd{generalized $Z$-(co)homology} theories
    \[ Z_*(\scr F) = {\rm HF}_*(\scr F; Z), \quad \text{or} \quad Z^*(\scr F) = {\rm HF}^*(\scr F; Z), \]
    at least when $Z$ has finite homotopy-support.
  \end{THM}

  \begin{COR}
    Stable $\pi_*(\scr F; \tau_{\le N-3}R) := {\rm HF}_*(\scr F; \tau_{\le N - 3} R)$ is well-defined, over $\tau_{\le N-3}(R)$.
  \end{COR}

  \begin{NB}
    This definition of $\pi_*$ really depends on the underlying ring $\tau_{\le N-3}R$.
  \end{NB}

  \begin{RMK}
    One can also define the Steenrod algebra $\scr A^*_Z(\tau_{\le N-3}R)$ with $Z$-coefficients in an analogous fashion, and this will act on $Z^*$-theory.
  \end{RMK}

  \subsection{Construction thereof from a pro-spectrum}\label{ssec:act} In this subsection, we prove the main results stated previously, namely {\sc Thms.~\ref{thm:act},~\ref{thm:gcoh}}, via the Cohen-Jones-Segal (CJS) construction \cite{CJS}, \cite[\sc\S 5,\S 7]{me}. There are two main difficulties which must be addressed, going beyond what the author wrote previously in \cite{me}:
  \begin{enumerate}
    \item Due to the indexing set $I$ potentially being unbounded in both directions, the CJS construction at best only produces a \emph{pro-spectrum}.
    \item More seriously, the fact that $\scr F$ is $N$-truncated means that part of the input required to run the CJS construction is missing. It will have to be filled in via obstruction theory, which is the reason for working over the Postnikov truncation $\tau_{\le N-3}R$, as opposed to $R$ itself.
  \end{enumerate}
  First, we recall the construction in the finite case, assuming $N = \infty$ (so no missing data).

  \begin{NOT}
    Define the Euclidean space $\mathbf F_{i,j}$ with $\{u : i < u \le j\}$-corners as $\mathbf E_{i,j} \times \mathbf R_{\ge 0}^{\{j\}}$ if $i < j$, and as $0$ if $i = j$. Note that $\partial^u \mathbf F_{i,j} = \mathbf E_{i,u} \times \mathbf F_{u,j}$ for all $i < u \le j$.
  \end{NOT}

  \begin{DEF}[CJS construction for $N = \infty$]
    Let $\scr F$ be an $R$-oriented ($\infty$-truncated) flow category, with the underlying indexing set $I$ being finite. Let the \emph{unstable} CJS construction associated to $\scr F$ be defined as the quotient
    \[ {\rm CJS}^u(\scr F; R) := \coprod_{i \in I} R_{\min I, i} \wedge \mathbf F_{i, \max I}^+ \; \Big/ \sim, \]
    where the equivalence relation $\sim$ identifies points in the $\partial^j$-boundary of the $i^{\rm th}$ summand, to points in the $j^{\rm th}$ summand, via the Pontrjagyn-Thom map
    \[ R_{\min I, i} \wedge \mathbf E_{i,j}^+ \wedge \mathbf F_{j, \max I}^+ \overset{{\rm id} \wedge {\rm PT} \wedge {\rm id}}{-\!\!\!-\!\!\!\longrightarrow} R_{\min I, i} \wedge R_{i,j} \wedge \mathbf F_{j, \max I}^+ \overset{\mu \wedge {\rm id}}\to R_{\min I, j} \wedge \mathbf F_{j, \max I}^+, \]
    which was constructed in {\sc Def.~\ref{def:pt}}. For $\sim$ to be transitive on the nose, the $\bb E_1$-map ${\rm PT} : \mathbf E_{*,*}^+ \to R_{*,*}$ should be strictly associative, which can always be arranged, e.g. by working in orthogonal spectra ($\bb E_1$-algebras can always be made strictly associative.) The (stable) \keywd{CJS construction} is obtained by formally desuspending the unstable one:
    \[ {\rm CJS}(\scr F; R) := {\rm CJS}^u(\scr F; R)[\mu(\max I) - {\rm rk\;} N_{\min I, \max I}]. \]
    This shift is best understood as making the ($\max I$)-summand equal to $R[\mu(\max I)]$, since the $i^{\rm th}$ point should have degree $\mu(i)$. The shifting is also required to make sure that the construction is invariant under stabilizations of the embedding of $\scr F$. Our result \cite[\sc Prop.\,5.7]{me} ensures that the construction does not depend on the contractible space of choices we made. There is also an analog of \cite[\sc Prop.\,5.8]{me}, which in our previous work we had only stated over the sphere spectrum, but easily generalizes:
  \end{DEF}

  \begin{PROP}[Recovering Floer (co)homology]\label{prop:rechf}
    The singular cohomology $H^*_R({\rm CJS}(\scr F; R); M)$ over $R$ coincides with ${\rm HF}^*(\scr F; M)$, and likewise for singular homology.
  \end{PROP}

  \begin{proof}
    This is entirely similar to the proof of \cite[\sc Prop.\,5.8]{me}.
  \end{proof}

  Next, we deal with the case that $N < \infty$, so that the Pontrjagyn-Thom $\bb E_1$-map $\mathbf E_{*,*}^+ \to R_{*,*}$ constructed in {\sc Def.~\ref{def:pt}} is only partially defined, when $\mu(i) - \mu(j) < N$. As a result, the equivalence relation in the definition of ${\rm CJS}^u$ is only partially defined. It turns out to be a bad idea to only impose this partial equivalence relation (i.e.~ its transitive closure), e.g. since {\sc Prop.~\ref{prop:rechf}} will not hold any longer. Instead, one must fill in the higher ${\rm PT}_{i,j}$ via obstruction theory:

  \begin{PROP}[Obstructions]\label{prop:obs}
    Assuming all the lower boundary strata of ${\rm PT}_{i,j}$ have been built, the obstruction to defining ${\rm PT}_{i,j}$ itself lies in $\pi_{\mu(i) - \mu(j) - 2} R$.
  \end{PROP}

  \begin{proof}
    The space $\mathbf E_{i,j}^+$ is built from the union of all its proper boundary strata by attaching precisely one cell, of dimension $\dim \mathbf E_{i,j}$. Hence, the obstruction element lies in $\pi_{\dim \mathbf E_{i,j} - 1} R_{i,j}$. Since $R_{i,j} = R[{\rm rk\;} N_{i,j}]$, this is the same as $\pi_d R$, where $d$ is given by
    \[ d = \dim \mathbf E_{i,j} - 1 - {\rm rk\;} {N_{i,j}} = \dim \scr M_{i,j} - 1 = \mu(i) - \mu(j) - 2, \]
    as desired.
  \end{proof}

  \begin{COR}
    If $\pi_* R = 0$ for all $* \ge N - 2$, then one can define the higher ${\rm PT}_{i,j}$ inductively by arbitrarily extending them from lower boundary strata. Moreover, given any $\partial D^k$-parameter family of such extensions, one can always fill it in to a $D^k$-parameter family.
  \end{COR}

  \begin{proof}
    The first part follows clearly from the proposition. For the second part, the obstructions to achieving the construction parametrically lie in $\pi_{\mu(i) - \mu(j) - 2 + k} R$, which is still zero.
  \end{proof}

  \begin{COR}
    For any ring $R$, one can always obtain a well-defined CJS construction ${\rm CJS}(\scr F; \tau_{\le N-3} R)$ over the Postnikov truncation $\tau_{\le N-3} R$.
  \end{COR}

  This deals with difficulty {\sc ii} from the introductory paragraph. Difficulty {\sc i} is taken care of by the following:

  \begin{PROP}[Finite range approximation]\label{prop:fra}
    Given any degree $d \in \bb Z$, the construction ${\rm CJS}(\scr F[i:j]; \tau_{\le N-2} R)$ applied to the restriction of the flow category $\scr F$ to just indices $i < u < j$ has cohomology group $H^d_R(-;M)$ independent of $i$ and $j$, provided that $i$ is small enough and $j$ is large enough. Moreover, their common value is equal to ${\rm HF}^d(\scr F; M)$. Additionally, given any Steenrod operation $\phi \in \scr A^r_M(\tau_{\le N-3} R)$, its action $\phi : H^d_R(-;M) \to H^{d+r}_R(-;M)$ on ${\rm CJS}(\scr F[i:j]; \tau_{\le N-3} R)$ is independent of $i$ and $j$ for sufficiently small $i$ and sufficiently large $j$.
  \end{PROP}

  \begin{proof}
    Decreasing $i$ by one (resp. increasing $j$ by one) changes ${\rm CJS}(\scr F[i:j]; \tau_{\le N-3} R)$ by attaching (resp. co-attaching), a $\mu(i)$-shifted (resp. $\mu(j)$-shifted) $\tau_{\le N-3}R$-cell. Since for each $k \in \bb Z$ there are only finitely elements in the preimage $\mu^{-1}(k)$, this means that as $i$ gets smaller and smaller, and $j$ gets larger and larger, one will have eventually exhausted all points of $\mu$ lying in any given finite range. As a result, its cohomology, as a module over the Steenrod algebra, does not change in that range.
  \end{proof}

  The result {\sc Thm.~\ref{thm:act}} now follows. Lastly, we note how this story parallels to generalized (co)homology theories, thereby proving {\sc Thm.~\ref{thm:gcoh}}.

  \begin{PROP}
    The analogous result to {\sc Prop.~\ref{prop:fra}} holds for generalized (co)homology theories $Z^*, Z_*$, given the finiteness condition {\sc Def.~\ref{def:fin}} is satisfied.
  \end{PROP}

  \begin{proof}
    This reduces to checking whether $Z^*(\tau_{\le N - 3}R)$ and $Z_*(\tau_{\le N - 3}R)$ are vanishing outside some finite range, which is precisely the content of the finiteness condition.
  \end{proof}

  \begin{RMK}[Floer pro-spectrum, and technical issues]\label{rmk:prosp}
    The segments ${\rm CJS}(\scr F[i:j]; \tau_{\le N-3}R)$ fit into a so-called module \keywd{pro-spectrum}, cf.~ \cite{CJS}. Sadly, no model structure on pro-spectra has been constructed, in which this pro-spectrum would be an invariant of flow categories modulo homotopy equivalences of flow categories. One would like the weak equivalences to be maps that induce isomorphisms on ordinary homology, while the model structures defined by Christensen-Isaksen \cite{ChrIs} are based on $\pi^*$-equivalences.
  \end{RMK}

  \begin{RMK}[Obstruction theory and Steenrod operations]\label{rmk:obs}
    The obstructions described in {\sc Prop.~\ref{prop:obs}} in particular give sufficient conditions for when one can upgrade the Cohen-Jones-Segal construction from a Postnikov truncation $\tau_{\le r} R$ to the next $\tau_{\le r + 1} R$. It would be interesting to know to what extent this also provides a necessary condition, since a priori this obstruction is built from choosing a specific cell structure for the module spectrum, as well as a particular $R$-orientation for the next moduli spaces.

    We take the liberty here to speculate a bit more on this subject. Another way to prove that a $\tau_{\le r} R$-module $M$ is not induced from $\tau_{\le r + 1} R$ is to show that $M$ has some nonzero cohomology operation $\alpha$ over $\tau_{\le r} R$, such that $\alpha$ is in the kernel of $\scr A^*(\tau_{\le r} R) \to \scr A^*(\tau_{\le r + 1} R)$. One wonders if this kind of Steenrod algebraic obstruction is at all related to the obstructions we got in {\sc Prop.~\ref{prop:obs}}.

    We also note that the obstructions of {\sc Prop.~\ref{prop:obs}}, in the case of monotone Lagrangians (to be discussed in {\sc\S\ref{sec:apps}} shortly), should express some kind of real (homotopy) Gromov-Witten invariant. Classically, the curvature phenomena one obtains for minimal Maslov number $N_\mu = 2$ translate algebraically into $d^2 = c \cdot {\rm id}$, where $c$ is a count of holomorphic disks with boundary on the Lagrangian, hence exhibiting the obstruction to making ${\rm CF}_*$ a chain complex. If our speculations above are in any sense correct, then the real Gromov-Witten theory of the Lagrangian would be somehow reflected in the Steenrod algebra action. It would also show that the Floer homotopy type is homotopically ``interesting,'' as it does not extend to the next Postnikov truncation.
  \end{RMK}

\newpage
\section{Applications to Monotone Lagrangian theory}\label{sec:apps}

  \subsection{Symplectic setup}\label{ssec:setup} Although our theory no doubt admits many generalizations, in the present exposition we content ourselves with studying data of the following particular kind:
  \begin{enumerate}
    \item A closed symplectic manifold $(M, \omega)$, which is \emph{monotone}, i.e.~ $\int_\omega$ is a real positive multiple of $2c_1$, as homomorphisms $\pi_2 (M) \to \bb R$.
    \item A closed embedded Lagrangian $L \subset M$, which is \emph{monotone}, i.e.~ $\int_\omega$ is a real positive multiple of $\mu$, as homomorphisms $\pi_2 (M,L) \to \bb R$. (This monotonicity condition actually implies the former.)
    \item These are required to satisfy $\pi_1(M,L) = 0$ and $\pi_2(M,L) = \bb Z$.
    \item Let $N_\mu$ be the positive generator of the image of $\mu : \pi_2(M,L) \to \bb Z$, and $N_{2c_1}$ be the positive generator of the image of $2c_1 : \pi_2(M) \to \bb Z$. One automatically has $N_\mu \le N_{2c_1}$, by the action of $\pi_2(M)$ on $\pi_2(M,L)$:
    \[ \mu(A \# u) = 2c_1(A) + \mu(u), \qquad \forall A \in \pi_2(M), u \in \pi_2(M,L). \]
    \!{\em We require that} $N_\mu \ge 3$ (in fact, $N_\mu \ge 5$ in order for our theory to produce results that go beyond mere homology.)
  \end{enumerate}
  We will impose one extra topological constraint later, when we discuss index theory, cf.~ {\sc Def. \ref{def:ubr}}, in order to get a $U$-structure on our $N_\mu$-truncated flow category. The monotonicity conditions ensure there is no bubbling for sufficiently low-dimensional moduli spaces, while the topological constraints guarantee that every Hamiltonian chord can be homotoped to lie in $L$, and that the space of such homotopies is a $\bb Z$-torsor. The latter will ensure there are only finitely many chords of the same index.

  \begin{EXP}
    The standard inclusion $\bb{RP}^n \subset \bb{CP}^n$ satisfies these conditions when $n \ge 2$, since $N_\mu = n + 1$. We will later need to enforce $n$ to be odd, for the $U$-structure to exist.
  \end{EXP}

  \begin{DEF}[Hamiltonians]
    A \keywd{non-degenerate} time-dependent Hamiltonian is a smooth function $H : M \times [0,1] \to \bb R$, often denoted $H_t : M \to \bb R$, such that the time-dependent flow $\Phi_t$ associated to it, defined by
    \[ \omega(v, \dot \Phi_t(x)) = dH_t(v), \qquad \forall t \in [0,1], x \in M, v \in T_x M. \]
    satisfies $\Phi_1(L) \pitchfork L$. A more lax condition is that $H$ is \keywd{cleanly degenerate}, which means that $\Phi_1(L)$ intersects $L$ \emph{cleanly}, i.e.~ $C := \Phi_1(L) \cap L$ is a smooth submanifold of $M$, satisfying $TC = TL \cap \Phi_1(TL)$ as linear subspaces of $TM$.
  \end{DEF}

  \begin{DEF}[Complex structures]
    A time-dependent almost complex structure $J : [0,1] \to C^\infty(\scr{End}(TM))$, often denoted $J_t \in C^\infty(\scr{End}(TM))$, is a smooth 1-parameter family of tensors satisfying $J^2 = -{\rm id}$. It is said to be \keywd{$\omega$-compatible} when $g_t(-,-) := \omega(-,J_t -)$ is a Riemannian metric on $M$ for all $t \in [0,1]$.
  \end{DEF}

  \begin{DEF}[Floer equation]
    Assume we are given a \emph{non-degenerate} Hamiltonian $H_t$ and $\omega$-compatible almost-complex structures $J_t$. Let also $x_0, y_0 \in \Phi_1(L) \cap L$ be two intersection points, which can be extended to \keywd{Hamiltonian chords} $x, y : (I, \partial I) \to (M,L)$ via $x_t := \Phi_t(x)$. Then, the \keywd{Floer equation} for a $C^\infty$ (or more generally $W^{1,2}_{\rm loc}$-class) map $u : (\bb R_s \times [0,1]_t, \bb R_s \times \{0,1\}) \to (M,L)$ takes the form
    \[\numeq\label{eq:floer} \partial_s u + J_t(\partial_t u - X_{H_t}|_{u}) = 0 \qquad \text{on } u^*TM, \]
    where $X_{H_t}$ is the unique time-dependent vector field satisfying $\omega(-, X_{H_t}) = dH_t$. For generic $J_t$, the space of solutions carries the structure of a manifold, cf.~ \cite[\sc Thm.\,3.2]{Albers}.
    
    To understand the dimension of its various connected components, we need to introduce a few more concepts: a \emph{homotopy-cap} for a Hamiltonian chord $x$ is a (homotopy class of) a choice of homotopy taking $x$ inside $L$. Since $\pi_1(M,L) = 0$, every Hamiltonian chord admits a homotopy-cap, and moreover the set of such homotopy-caps is a $\bb Z = \pi_2(M,L)$-torsor. Hamiltonian chords equipped with a choice of homotopy-cap are called \keywd{capped chords}, and denoted $\tilde x, \tilde y$, etc.~ Denote the set of all such capped chords by $I_H$.

    Now, a solution to the Floer equation is said to be \emph{cap-compatible} with respect to two choices of caps for $x$ and $y$ if gluing the chosen cap for $x$ to the strip $u$ induces the chosen cap for $y$. Therefore, we can restrict to the space of solutions between $\tilde x$ and $\tilde y$ that are cap-compatible. In this situation, there exists a function $\mu : I_H \to \bb Z$, called the \keywd{Maslov grading}, such that the dimension of the space of cap-compatible solutions from $\tilde x$ to $\tilde y$ has dimension $\mu(\tilde x) - \mu(\tilde y)$. This Maslov grading is a priori only well-defined up to a shift, although it is possible to absolutize it canonically, cf.~ \cite{Albers}. We also note that if $\tilde x'$ is obtained by adding a certain element $A \in \pi_2(M,L)$ to the homotopy-cap of $\tilde x$, then one has
    \[ \mu(\tilde x') - \mu(\tilde x) = \mu(A). \]
    This is always a multiple of $N_\mu$, and in fact is equal to $A \cdot N_\mu$, if we identify $A \in \pi_2(M,L) = \bb Z$ with an integer.

    The space of cap-compatible solutions from $\tilde x$ to $\tilde y$ admits a natural $\bb R$-action, in view of the $s$-translation symmetry of the equation. This action is free and proper for all $\tilde x, \tilde y$, \emph{except} for the case that $\tilde x = \tilde y$, in which the space is a point (given by the constant solution). This inconvenience will be important in establishing precisely the truncation level of our flow category. For $\tilde x \neq \tilde y$, we denote the space of cap-compatible solutions modulo this $\bb R$-action by $\scr M_{\tilde x, \tilde y}^\circ$, where the little circle indicates that this is the interior of some suitably compactified version.
  \end{DEF}

  \subsection{The Lagrangian flow category}\label{ssec:lagflow} Now, we construct the flow category associated to the data assumed in the previous subsection.

  \begin{THM}[Adding broken strips]\label{thm:brok} (due to Large \cite{Large}.)
    For all $\tilde x \neq \tilde y$, there is a partial compactification $\scr M_{\tilde x,\tilde y}$ of $\scr M_{\tilde x, \tilde y}^\circ$ obtained set-theoretically by adding all sequences of broken trajectories starting at $\tilde x$ and ending at $\tilde y$ (and disallowing the exceptional constant strip). This compactification can be endowed with the structure of a smooth manifold with $I_H$-corners, labelling the places where the path breaks, canonically up to an appropriate choice of gluing parameter.
  \end{THM}

  \begin{proof}
    The set $\scr M_{\tilde x, \tilde y}$ is simply defined as the disjoint union of all products $\scr M_{\tilde x, \tilde z_1} \times \cdots \times \scr M_{\tilde z_k, \tilde y}$ (disallowing $\tilde x = \tilde z_1$, $\tilde z_i = \tilde z_{i+1}$, $\tilde z_k = \tilde y$) coming from intermediary sequences of capped chords. The topology is given by the standard Gromov convergence of nets. That this is a topological manifold with $I_H$-corners follows from standard gluing arguments. The difficult part is the smooth structure, which is constructed by Large \cite[{\sc\S6}]{Large}, upon choosing an (appropriate) gluing profile. Even though the setup of his paper is Liouville manifolds, the proof of the smoothness result does not use the Liouville structure. This is because the well-definition of the smooth structure boils down to checking the smoothness of a certain transition function, which is a purely local statement that follows from certain exponential decay estimates, due originally to Fukaya-Oh-Ohta-Ono \cite[{\sc\S6}]{FOOOexp}. We should also mention that, since we never add sphere or disk bubbles in our compactifications, we need not worry about the technical difficulties in making the gluing parameter agree with the more natural one coming from Gromov-Witten theory.
  \end{proof}

  \begin{THM}[Lagrangian Floer flow category]\label{thm:lagflow}
    From the data of $\scr M_{\tilde x, \tilde y}$, one can construct an $N_\mu$-truncated flow category $\scr F(M,L)$.
  \end{THM}

  \begin{proof}
    First, we must endow the set $I_H$ of capped chords with a total order on it. To do this, we use the \keywd{action filtration}, defined as
    \[ \scr A(\tilde x) := \int_{\bb D^2} \tilde x^* \omega - \int_{[0,1]} H_t(x(t))\, dt,  \]
    where $\bb D^2$ is the cap, and $\tilde x^* \omega$ is the pull-back of the symplectic form to the cap. This is independent of the choice of cap in its homotopy-cap class, by Stokes' theorem, and the fact that $L$ is Lagrangian. This action has the property that solutions to the Floer equation always decrease the action, except in the case that the strip is constant. Hence, \emph{we define} the total order on $I_H$ by setting $\tilde x < \tilde y$ whenever $\scr A(\tilde x) > \scr A(\tilde y)$. This is potentially ambiguous when we have distinct capped chords with the same action, in which case we order them arbitrarily. The action satisfies the property that 
    \[ \scr A(\tilde x') - \scr A(\tilde x) = \lambda \cdot (\mu(\tilde x') - \mu(\tilde x)), \quad \text{for some constant } \lambda > 0, \]
    which is an integer multiple of a positive real multiple $\lambda N_\mu$ of $N_\mu$, for any two capped chords $\tilde x', \tilde x$ representing the same chord $x$. This is due to the monotonicity assumption. A corollary of this observation is that $(I_H, <)$ is embeddable in $(\bb Z, <)$, because modulo $\lambda N_\mu$ there are only finitely many possible values of the action (since $\Phi_1(L) \cap L$ is compact). A similar argument shows that $\mu^{-1}(d)$ is finite for any $d \in \bb Z$.

    In light of this, and {\sc Thm.~\ref{thm:brok}} above due to Large, the only non-trivial part remaining to be established is the \emph{compactness of} $\scr M_{\tilde x, \tilde y}$ for $\mu(\tilde x) - \mu(\tilde y) < N_\mu$. By Gromov compactness, any sequence of solutions to the Floer equation has a convergent subsequence up to potentially broken strips, disk bubbles, and sphere bubbles. The effect of gluing a disk bubble on the Maslov index is adding a positive integer multiple of $N_\mu$, and likewise a positive integer multiple of $N_{2c_1}$ for a sphere bubble. Since $N_\mu \le N_{2c_1}$, this means that the Maslov grading difference of any strip obtained from gluing a bubble is at least $N_\mu$. (The edge case is when we glue a simple disk bubble, i.e.~ of Maslov index precisely $N_\mu$, to a constant strip.) So if $\mu(\tilde x) - \mu(\tilde y) < N_\mu$, then bubbles are excluded, and broken strips are the only possible phenomena that could break compactness. Since we added those in by {\sc Thm.~\ref{thm:brok}}, this ensures the compactness of $\scr M_{\tilde x, \tilde y}$, and completes the proof.
  \end{proof}

  \begin{RMK}
    The potential ambiguity in defining the ordering on $I_H$ for capped chords of the same action is harmless, since the independence of choices argument in {\sc\S\ref{ssec:indch}} will in particular also take care of this.
  \end{RMK}

  \begin{RMK}[Potential improvement]\label{rmk:impr}
    It is likely that the truncation degree can be improved from $N_\mu$ to $2N_\mu$, by the following trick, which was communicated to me privately by Mohammed Abouzaid. When gluing on a disk bubble, this can occur on either side of the strip. The new boundary strata that one needs to add in the range $\mu(\tilde x) - \mu(\tilde y) < 2N_\mu$ are those given precisely by gluing on a simple bubble, on either side. If these strata could be glued together, they would cancel and one would still get a traditional flow category, without any curvature phenomena. This gluing idea does not work as stated, but rather one should instead insert a \emph{cancelling cobordism}, by choosing paths generically inside $L$ connecting points on the $0$-side of the strip to points on the $1$-side of the strip, and considering the moduli space of disks with boundary on $L$, and a marked point living on this path. However, this will not actually improve our applications to $\bb{RP}^n \subset \bb{CP}^n$ of {\sc\S\ref{ssec:proj}}, so we do not pursue this idea further in the present paper.
  \end{RMK}

  \subsection{Index theory}\label{ssec:ind} In this subsection, we set up the index theory necessary to get a purely topological criterion for when $\scr F(M,L)$ admits a $U$-structure, cf.~ {\sc Def.~\ref{def:ubr}} below. We should mention that earlier work of Porcelli \cite{Porcelli}, and later Porcelli-Smith \cite{PoSm2}, use similar ideas. To understand when $U$-structures exist, one is naturally led to studying the parameter-space of all linearized Floer-type operators:

  \begin{DEF}[Category of non-degenerate APS ends]\label{def:aps}
    Slightly generalizing the (linearization of the) Floer equation \ref{eq:floer}, one can consider real Cauchy-Riemann operators on the strip $\bb R_s \times [0,1]_t$ with nondegenerate Atiyah-Patodi-Singer (APS) type asymptotic ends, i.e.~ the following data
    \begin{enumerate}
      \item A complex virtual bundle $E$ of rank 0 on $\bb R_s \times [0,1]_t$. We write $E$ for an actual vector bundle representing this virtual bundle, and understand this to be well-defined up to stabilization by trivial copies of $\bb C^n$;
      \item A totally real sub-bundle $F$ of the restriction of $E$ to $\bb R_s \times \{0,1\}$. Under stabilization of $E$ by $\bb C^n$, one correspondingly stabilizes $F$ by $\bb R^n$.
      \item\label{itm:linop} A $1^{\rm st}$-order linear differential operator $C^\infty(E,F) \to C^\infty(E)$ of the form $\partial_s + J\partial_t + A_s$, where $A_s$ is an $s$-parameter family of arbitrary $0^{\rm th}$ order terms. Under stabilization of $E$ and $F$, one subsequently stabilizes the $A_s$ by zero.
      \item Such that $A_{\pm \infty} := \lim_{s \to \pm \infty} A_s$ exist, the limit converges exponentially in all $C^k$ norms, and most crucially the operators $J \partial_t + A_{\pm \infty}$ on the interval $[0,1]$ are invertible (this is the so-called ``APS non-degeneracy condition.'') Note that stabilizations do not affect this condition, because $\bb R^n$ and $e^J \bb R^n$ inside $\bb C^n$ intersect transversely.
    \end{enumerate}
    There is a (topological) $\bb A_\infty$-category whose objects are given by APS-type non-degenerate operators on the interval $[0,1]$, as in the last item above, and whose morphisms are real Cauchy-Riemann operators on the strip, connecting two such APS-type operators at $\pm \infty$. The $\bb A_\infty$-composition is given by gluing of operators. (There are some contractible spaces of choices involved here, which is why this is an $\bb A_\infty$-category.) We call this category $\scr {APS}$.
  \end{DEF}

  \begin{PROP}
    The $\bb A_\infty$-category $\scr {APS}$ is in fact an $\bb A_\infty$-groupoid, and it is equivalent to the $\bb E_1 = \bb A_\infty$-algebra $\Omega(BO \times^h_{BU} BO)$, via path concatenation.
  \end{PROP}

  \begin{proof}
    Any two operators can be connected by a strip. By considering the reverse strip, and concatenating the two, it is not hard to see that the result can be deformed back to the identity. Thus, all objects are isomorphic, and hence the $\bb A_\infty$-category is equivalent to a skeleton, upon choosing some preferred base object. It remains to see what the $\bb E_1$-algebra of endomorphisms of this object is. The $0^{\rm th}$-order term of the operator lives in a contractible space of choices, and can safely be ignored. All the data that remains is the two bundles $E$ and $F$. We think of a strip as a path of intervals $[0,1]$, equipped with a complex bundle $E$, and two totally-real subspaces $F_0, F_1$ at the ends. Trivializing $E$ rel an endpoint (which can be done up to contractible space of choices), one can transport one of the totally real subspaces, e.g. $F_0$, to the other end, resulting in two different totally real subspaces $F_0, F_1 \subset E_1$. This is precisely the data of two real bundles, with a choice of isomorphism of their complexifications, i.e.~ a point in the fiber-product $BO \times^h_{BU} BO$. The result now follows.
  \end{proof}

  \begin{PROP}\label{prop:ouo}
    One has an isomorphism $\Omega(BO \times^h_{BU} BO) \cong O \times \Omega(U/O)$ of $\bb E_1$ (and in fact $\bb E_\infty$) groups.
  \end{PROP}

  \begin{proof}
    First, $\Omega(BO \times^h_{BU} BO) \cong O \times^h_U O$, which is purely formal due to the fact that in general $\Omega X = * \times^h_X *$, and homotopy-limits always commute. There is no ambiguity in whether we use the $\bb E_1$-structure coming from the loops, or the $\bb E_1$-structure coming from the $\bb E_\infty$-structure on $O$ and $U$, by the Eckmann-Hilton argument. Further, note that the difference endomorphism $(x, y) \mapsto (x, x - y)$ on $O \times O$ is an equivalence (since $O$ is group-like). Since the fiber product $O \times^h_U O$ is the homotopy-fiber of the difference map $O \times O \overset - \to O \to U$, it follows that $O \times^h_U O \cong O \times \Omega(U/O)$, where the $O$-component is just the projection onto the first factor, while the other is the witness in $\Omega(U/O) = {\rm Fib}(O \to U)$ of the fact that the two real bundles have the same complexification.
  \end{proof}

  \begin{DEF}[The index map]\label{def:ind}
    There is an associated \keywd{index map} 
    \[ {\rm Ind} : \scr{APS} \to \Omega^\infty KO,\]
    given by taking the index of the Cauchy Riemann operator with respect to the $L^2$-completion (recall, the space of Fredholm operators is a classifying space for K-theory, cf.~ Atiyah-J\"anich \cite{At,JanK}.) 
    
    The index is invariant under the stabilization in {\sc Def.~\ref{def:aps}-\ref{itm:linop}}, because the stabilization of $A_s$ by zero correspondingly stabilizes the Fredholm operator by an invertible operator. This process is an $\bb E_1$ map because concatenation of operators on strips results in the index bundles adding up. In light of the two propositions above, there is also an $\bb E_\infty$-structure on both sides, and the map also respects this, because the index is compatible with addition of bundles. By the Eckmann-Hilton argument, the two different operadic structures on this map must in fact coincide up to contractible homotopy, since the $\bb E_1$ path concatenation and the $\bb E_\infty$ addition of bundles respect each other.
  \end{DEF}

  \begin{PROP}
    The contribution of the $O$-factor (via the isomorphism proven above in {\sc Prop.~\ref{prop:ouo}}) under the index map is zero.
  \end{PROP}

  \begin{proof}
    Put differently, if the projection to the $U/O$ factor is zero, we must show that the index is canonically zero. This has already been done in the literature, under the slightly different language of ``framed brane structures,'' cf.~ \cite{Large} and \cite[\sc Thm.\,1.7]{Blakey}. The idea is that one can instead glue a standard \emph{thimble} at one end, and because the difference element in $\Omega(U/O)$ is zero, one can canonically deform the operator on the glued thimble + strip, to get back to the standard thimble operator, hence witnessing that the index is zero. This can be done compatibly with strip concatenation and vector bundle addition.
  \end{proof}

  \begin{DEF}
    The \keywd{reduced index map} ${\rm Ind} : \Omega(U/O) \to \Omega^\infty KO$ is the old index map, factored through projection onto the second factor under the isomorphism of {\sc Prop.~\ref{prop:ouo}}.
  \end{DEF}

  \begin{THM}[Index theorem]\label{thm:ind}
    The reduced index map is an equivalence of $\bb E_\infty$-groups.
  \end{THM}

  \begin{proof}
    It suffices to prove that this is a split surjection, since their $\pi_*$ (which are either $0$, $\bb Z/2$ or $\bb Z$) have the remarkable property that they can only be a summand of themselves via the identity. Hence, all we have to do is provide a right-inverse. To this end, pick some element in the left-hand side whose index has dimension 1. (Such an element exists, e.g. consider the obvious holomorphic disk bounded by a bigon in $\bb C$, which has Maslov index 1.) Call this operator $D_0$. Then, construct a map $\Omega^\infty KO \to \Omega(U/O)$ that sends a bundle $V$ to the parameter-family of operators $V \otimes D_0$. The latter makes sense because locally it is just a direct sum of copies of $D_0$. The index of this tensor product is ${\rm Ind}(D_0) \otimes V \cong \bb R \otimes V \cong V$, again because of the additivity of the index. This proves the claim.
  \end{proof}

  \subsection{\texorpdfstring{$U$}{U}-branes and \texorpdfstring{$U$}{U}-structures}\label{ssec:ustruc} Now that we have a somewhat algebraic expression for the index, it is natural to wonder under what conditions one can lift the index $\bb E_1$-map 
  \[ \scr M_{*,*} \to \scr {APS} \overset{\rm Ind}\longrightarrow \Omega^\infty KO \overset{\rm forg\; dim}\longrightarrow BO\]
  to $BU$, hence obtaining a $U$-structure on $\scr F(M,L)$. To understand this, one must choose an isomorphism for every object in $\scr{APS}$, connecting it back to to the basepoint (which one may recall is the procedure for producing a skeleton, which allowed us to say that $\scr {APS} \cong O \times \Omega(U/O)$.) After pre/post composing with these isomorphisms, we get a new map $\scr M_{*,*} \to O \times \Omega(U/O) \subset \scr {APS}$, given by sending a solution $u$ of the Floer equation to the gluing thereof with some other abstract strips $v_{\tilde x}, v_{\tilde y}$ connecting the APS-type ends $\scr H_{\tilde x}$ and $\scr H_{\tilde y}$ of $\tilde x$ and $\tilde y$, to some base APS-end $\scr H_0$:
  \[ \hat u : \scr H_0 \overset{v_{\tilde x}}\rightsquigarrow \scr H_{\tilde x} \overset u \rightsquigarrow \scr H_{\tilde y} \overset{v_{\tilde y}^{-1}}\rightsquigarrow \scr H_0 \]
  inside $\scr{APS}$. The index changes as a result of this, by adding the indices of $v_{\tilde x}$ and $v_{\tilde y}^{-1}$. Thus,
  \[\numeq\label{eq:indcorr} {\rm Ind}(\hat u) = {\rm Ind}(v_{\tilde x}) + {\rm Ind}(u) - {\rm Ind}(v_{\tilde y}). \]
  The crucial observation is that the terms ${\rm Ind}(v_{\tilde x})$ and $- {\rm Ind}(v_{\tilde y})$ that we have chosen are \emph{constant} over the moduli space $\scr M_{\tilde x, \tilde y}$, since they only depend on the asymptotic APS-type ends $\scr H_0, \tilde x, \tilde y$. Thus, when forgetting the dimension via the map $\Omega^\infty KO \to BO$, the result does not change, despite introducing these extra paths $v_{\tilde x}, v_{\tilde y}^{-1}$.

  \begin{RMK}
    It might be confusing how replacing an $\bb A_\infty$-category with an equivalent one (in our case, the skeleton), changes the index. This is due to the fact that equivalences of categories are not truly isomorphisms, but only satisfy $g \circ f = {\rm id}$ and $f \circ g = {\rm id}$ up to natural isomorphisms in the category we are working with. In $\Omega^\infty KO$, every morphism is an isomorphism, because adding a virtual bundle is always a reversible process. Nevertheless, we were lucky enough that the extra index bundles that show up are constant, and so when forgetting the dimension they become zero. We also point out that, had one chosen different isomorphisms $v_{\tilde x}'$ between $\scr H_0$ and $\scr H_{\tilde x}$, or perhaps a different $\scr H_0'$ altogether, the result would still not change after forgetting the dimension of the virtual bundle.
  \end{RMK}

  We have thus proven the following:

  \begin{PROP}
    The (dimensionless) index map $\scr M_{*,*} \to BO$ can be computed as the composite
    \[ \scr M_{*,*} \to \Omega(U/O) \overset\sim \to \Omega^\infty KO \to BO, \]
    where the first map is the difference element of the two Lagrangian subspaces on the two ends of the strip, the second one comes from our Index Theorem {\sc\ref{thm:ind}}, and the third one is the map that forgets the dimension.
  \end{PROP}

  However, in light of the previous remark, the dimensions of the index bundles cannot be computed in this way, and one must use the correction terms appearing in \ref{eq:indcorr}. This is essentially what the Conley-Zehnder index does, although we do not explain this further here. Finally, the key to obtaining $U$-structures is the following result:

  \begin{PROP}\label{prop:indbott}
    There is a commutative diagram of $\bb E_1$-maps
    \[ \begin{tikzcd}
      \Omega U \rar["\sim"]\rar[swap, "\rm Bott"]\dar & \Omega^\infty KU \dar \rar & BU \dar \\
      \Omega (U/O) \rar["\sim"]\rar[swap, "\rm Ind"] & \Omega^\infty KO \rar & BO
    \end{tikzcd} \]
    where the unlabelled maps are the straightforward ones.
  \end{PROP}

  \begin{proof}
    The right square is clearly commutative. The key to proving the commutativity of the left square is to actually prove the commutativity where ${\rm Bott}$ and ${\rm Ind}$ are replaced by their inverses, which have explicit algebraic formulas. Namely ${\rm Bott}^{-1}$ is given by viewing $\Omega U$ as $\Omega^2 BU$, and sending a virtual vector space $V$ to $V \otimes [\scr O(1) - 1]$, where $\scr O(1)$ is the anti-tautological bundle on $S^2$. The bottom map ${\rm Ind}^{-1}$, as explained in the proof of the Index Theorem {\sc\ref{thm:ind}}, is given by sending a virtual vector space $V$ to $V \otimes D_0$. Since the tensor product commutes with direct sums, it suffices to prove the commutativity of the reverse diagram for $V = \bb C$. Two points are path-connected in $\Omega^\infty KO$ if and only if they have the same dimension, so the proof further reduces to computing the index numerically of the strip associated to $\scr O_\bb C(1)$, and showing that it is equal to $2$. 
    
    To this end, we use the formula for the index of a strip as the winding number of the Lagrangian around the boundary of the strip, upon symplectically trivializing the complex bundle over it. The complex line bundle $\scr O(1)$ over the sphere can be thought of as a complex bundle over the disk, trivialized on the boundary. There is another trivialization of this $\scr O(1)$ from the contractibility of the disk instead, and the induced difference on the boundary is the generator of $\pi_1 U$. If we track the Lagrangian, we can see that it has winding number $2$ (e.g. imagine the line inside $\bb R^2 \cong \bb C$, rotating $360^\circ$; it intersects the horizontal direction exactly twice.) This concludes our proof.
  \end{proof}

  \begin{RMK}
    We note here that the homotopy we produced witnessing the commutativity of the left square is quite non-canonical, since we invoked the fact that two points in $\Omega^\infty KO$ are connected by a path if and only if they have the same dimension. Hence, there is a non-contractible space of choices involved in this. However, one can just pick one and stick with it. We suspect that the argument can be refined to get a canonical such choice, but this does not really matter for the purposes of this paper, since all we care about is existence of a $U$-structure.
  \end{RMK}

  As a corollary, we get

  \begin{THM}[$U$-structures in Lagragian theory]\label{thm:ustruc}
    The space $L \times^h_M L$ of paths in $M$ with endpoints in $L$ has a map $L \times^h_M L \to U/O$ given by forming the difference element of the two Lagrangian subspaces. Its post-composition with the connecting map $U/O \to BO$ arising from the principal fibration $O \hookrightarrow U \to U/O$ recovers the difference $[{\rm pr}_1^* TL] - [{\rm pr}_2^* TL]$ as virtual bundles, where ${\rm pr}_i : L \times^h_M L \to L$ are the projections onto the two factors. Any null-homotopy of this map $L \times^h_M L \to BO$ induces a $U$-structure on $\scr F(M,L)$. {\sc(Def.~\ref{def:g-flow})}
  \end{THM}

  \begin{DEF}[$U$-branes]\label{def:ubr}
    Such a null-homotopy is called a \keywd{$U$-brane}. There is a natural diagonal inclusion $L \to L \times^h_M L$ given by the constant path, whose composition with the map to $U/O$ is already canonically zero. Hence, restricting any $U$-brane via the diagonal inclusion gives a based $S^1$-parameter family of maps $L \to BO$, i.e.~ gives a map $L \to O$. We say that the $U$-brane is \keywd{diagonal-respecting} if this map itself is null-homotopic. We call the consequent $U$-structure diagonal-respecting if it comes from such a diagonal-respecting null-homotopy.
  \end{DEF}

  \begin{RMK}
    The relevance of the diagonal-respecting condition is that later we want to ensure that ${\rm HF}^*(M, L)$ has the same Steenrod action as $H^*_{\rm sing}(L)$ under the Albers-PSS isomorphism, and not that of some Thom spectrum over it.
  \end{RMK}

  \begin{EXP}[A good source of $U$-branes]\label{exp:sourceu}
    Any extension $\widetilde{TL}$ of the virtual bundle $TL$ on $L$ to the whole of $M$ induces a diagonal-respecting $U$-brane. Indeed, given any path connecting two points of $L$ inside the ambient $M$, one can canonically trivialize (up to contractible choice) $\widetilde{TL}$ over it, resulting in a witness that ${\rm pr}_{1}^*TL - {\rm pr}_2^* TL$ is zero. This is diagonal-respecting because for constant paths, the induced map will be the identity. Thus, if $[TL]$ is in the image of $\widetilde{KO}(M) \to \widetilde{KO}(L)$, then $L$ admits a (not necessarily unique) diagonal-respecting $U$-brane.
  \end{EXP}

  \begin{proof}[Proof of {\sc Thm.~\ref{thm:ustruc}}]
    First we recall that, as mentioned previously in {\sc Rmk.~\ref{rmk:index}}, it is a result due to Porcelli-Smith \cite[\sc\S6.4]{PoSm2}, that the index bundles of the linearized Cauchy-Riemann operators agree with the tangent bundles to the moduli spaces (plus the extra copy of $\mathbb R$ lost after quotienting by translation), in a way that is compatible with gluing, modulo a contractible space of choices. Thus, it suffices to put $U$-structures on the index bundles which we have been studying in the previous subsection.

    The first claim that $L \times^h_M L \to U/O \to BO$ recovers the difference $[{\rm pr}_1^* TL] - [{\rm pr}_2^* TL]$ follows by unwinding the construction of the difference element in $U/O$. Indeed, recall from the proof of {\sc Prop.~\ref{prop:ouo}}, that this difference map is constructed by taking the difference of the two totally real sub-bundles, together with the witness that its complexification is zero. Its map to $BO$ just forgets the second part of the data, remembering only the difference.

    The second claim follows because a null-homotopy of the map $L \times^h_M L \to U/O \to BO$ is precisely the data of a lift $L \times^h_M L \to U$, which when fed into the machinery of {\sc Prop.~\ref{prop:indbott}} produces the desired $U$-structure.
  \end{proof}

  \begin{RMK}\label{rmk:constrbr}
    If one instead desires a $1$-structure, also known as a stable framing, one would have to nullhomotope $L \times_M^h L \to U/O$ itself. However, in the monotone case this is not possible, because in particular it would force the Maslov number $N_\mu$ to be zero. This would be better suited for studying exact Lagrangians instead. However, we will use a local version of this idea later, for Pozniak-type neighborhoods of clean intersections, where the theory locally behaves as though we were in the exact case. Likewise, the existence of any $U$-brane forces the Maslov number $N_\mu$ to be divisible by $2$, since the map on $H^1$ induced by $U \to U/O$ is multiplication by two.
  \end{RMK}

  \subsection{Independence of choices}\label{ssec:indch} We want to make sure that the resulting Steenrod algebra action on ${\rm HF}^*(M,L)$ is independent of the particular choices of time-dependent Hamiltonian and almost-complex structures used in its construction. On the level of (co)homology itself, the isomorphism is proven by using \keywd{continuation maps}, which given two data $(H_t^{\rm I}, J_t^{\rm I})$ and $(H_t^{\rm II}, J_t^{\rm II})$ one chooses an $s$-parameter interpolation $(H_t^s, J_t^s)$ that recovers the first for very small $s$, and the latter for very large $s$. Then, one counts solutions to the Floer equation \ref{eq:floer}, where $H$ and $J$ now also depend on the position $s$ on the strip. This leads to a chain map
  \[\numeq\label{eq:ctmap} {\rm CF}_*(M, L, H_t^{\rm I}, J_t^{\rm I}) \to {\rm CF}_*(M, L, H_t^{\rm II}, J_t^{\rm II}), \]
  which we call the continuation map induced by the family $(H_t^s, J_t^s)$. Of course, genericity of the $J$'s is once again required to meet the transversality criteria. This map is shown to be an isomorphism by exhibiting a chain-level null-homotopy between its pre/post composition with the opposite continuation map. Since it is already well-established that continuation maps give chain homotopy equivalences (even over $\bb Z$, since a $U$-structure in particular induces an $SO$-structure), we need not worry about upgrading the homotopies to the level of CJS constructions (although this would be nice). It suffices to prove the following.

  \begin{PROP}\label{prop:contflow}
    Assume $(M,L)$ have been endowed with a $U$-brane, so that a well-defined action $\scr A^*_A(\tau_{\le N_\mu - 3}MU)$ exists on the $A$-coefficient cohomology of either side of \ref{eq:ctmap}. Then, the isomorphisms on cohomology induced by the continuation maps \ref{eq:ctmap} respect the Steenrod algebra action.
  \end{PROP}

  \begin{proof}
    In \cite[\sc\S 6]{me}, we have explained how to construct continuation maps on the level of Cohen-Jones-Segal constructions, lifting the continuation maps on the chain level, in the case that both flow categories have finite indexing sets $I$ and $J$. Even though in our case the flow categories are potentially unbounded both below and above, this does not really present an issue. As pointed out in the proof of {\sc Thm.~\ref{thm:act}}, one can truncate the flow category $\scr F$ by only looking at indices $i < u < j$, for sufficiently small $i$ and sufficiently large $j$, and the underlying (co)homology will be independent in any finite range. 
    
    Hence, it remains to appropriately define and actually construct an $N_\mu$-truncated continuation flow category from $\scr F(M,L, H_t^{\rm I}, J_t^{\rm I})$ to $\scr F(M,L, H_t^{\rm II}, J_t^{\rm II})$, and explain how the techniques of \cite[\sc\S 6]{me} apply to the truncated case. An $N$-truncated \keywd{continuation flow category} between two $N$-truncated flow categories $\scr F_I$ and $\scr F_J$ indexed on finite totally ordered sets $I$ and $J$, is another $N$-truncated flow category $\scr F_{I \sqcup J}$ indexed on $I \sqcup J$ (with $i < j$ for all $i \in I, j \in J$), which when restricted to $I$ and $J$, respectively, recovers $\scr F_I$ and $\scr F_J$, with the notable exception that the grading $\tilde \mu$ on $\scr F_{I,J}$ is given by $\tilde \mu(i) = \mu(i)+1$ for $i \in I$, and $\tilde \mu(j) = \mu(j)$ (this is because there is no $\bb R$-translation symmetry, and so moduli spaces have one larger dimension as a result of not quotienting by any $\bb R$-action.) The continuation map is built as the connecting homomorphism in the exact triangle associated to the cofiber sequence
    \[ {\rm CJS}(\scr F_J) \hookrightarrow {\rm CJS}(\scr F_{I \sqcup J}) \twoheadrightarrow {\rm CJS}(\scr F_J)[1]. \]
    See \cite[\sc\S 6]{me} for more details.

    So we must construct compact moduli spaces $\scr M_{\tilde x^{\rm I}, \tilde y^{\rm II}}$ with the obvious boundary compatibility conditions, as long as $\tilde \mu(\tilde x^{\rm I}) - \tilde \mu(\tilde y^{\rm II}) < N_\mu$. As in {\sc Thm.~\ref{thm:brok}} due to Large \cite{Large}, broken strips (with respect to either end I or II) can be added in order to form a partial compactification of the open moduli space. These are compact once again because any disk bubble or sphere bubble decreases the Maslov difference by $N_\mu$, leading to a strip of negative Maslov difference, which is impossible. This concludes our proof.
  \end{proof}

  As a corollary, we have the desired

  \begin{THM}[Action is well-defined]\label{thm:welldef}
    Given data as presented in the beginning of {\sc \S\ref{ssec:setup}}, endowed with a $U$-brane, cf.~ {\sc Def.~\ref{def:ubr}}, there is a well-defined Steenrod algebra action 
    \[ \scr A^*_A (\tau_{\le N_\mu - 3} MU) \; \circlearrowright \; {\rm HF}^*(M,L;A),\]
    on Floer cohomology, independent of the choices of Hamiltonian and almost complex structures, respectively. 
  \end{THM}

  With little extra work, we also obtain well-defined generalized cohomology theories for modules with finite homotopy-support:

  \begin{THM}[Generalized (co)homology is well-defined]\label{thm:welldefZ}
    With the same hypotheses as in the beginning of {\sc\S\ref{ssec:setup}}, a $U$-brane on the Lagrangian allows one to define generalized (co)homology theories
    \[ {\rm HF}^*(M,L;Z), \quad \text{and} \quad {\rm HF}_*(M,L;Z), \]
    at least when $Z$ is a $\tau_{\le N_\mu - 3} MU$-module with finite homotopy-support (cf.~ {\sc Def.~\ref{def:fin}}). The cohomology has a well-defined action of the Steenrod algebra $\scr A^*_Z(\tau_{\le N_\mu - 3} MU)$, and there are both homological and cohomological Atiyah-Hirzebruch spectral sequences converging strongly to the generalized (co)homology just defined. These are independent of the choices of Hamiltonian and almost-complex structure.
  \end{THM}

  \begin{proof}
    Again, the construction is very similar, so we omit the details. The finite homotopy-support condition on $Z$ once again ensures that we can use finite-range approximation to reduce the construction of the continuation flow from infinite to finite ranged flow categories. To show the induced continuation map on $Z$-(co)homology is well-defined, we actually rely on the Atiyah-Hirzebruch spectral sequence. Indeed, there is a finite length filtration on $Z$, coming from its Postnikov tower, with graded quotients given by Eilenberg-MacLane spectra over $R = \tau_{\le N_\mu - 3} MU$. This induces the desired spectral sequence, and from the fact that continuation maps induce isomorphisms on the $E_2$-page (via the result on ordinary (co)homology established above), it follows by the 5-lemma that it also induces an isomorphism on generalized $Z$-(co)homology.
  \end{proof}

  \begin{RMK}[Periodicity]
    The action of $\pi_2(M,L) = \bb Z$ on the capped chords results in an isomorphism ${\rm HF}^* \cong {\rm HF}^{*+N_\mu}$. This is easily seen to respect the Steenrod algebra action, since the action extends to the flow category.
  \end{RMK}

  \subsection{The Oh-Pozniak spectral sequence}\label{ssec:ohpoz} Now that we have a well-defined invariant, we want some tools to make it computable. It is natural therefore to look at already existing methods, and show that they are compatible with the Steenrod algebra action (and/or the generalized cohomology).
  
  In \cite[\sc Thm.\,iv]{Oh96}, Oh constructed a spectral sequence from the singular cohomology of $L$, to its Floer cohomology. This can be used to recover Albers' PSS isomorphism \cite{Albers}. Pozniak \cite{Poz} shows that if two exact Lagrangians intersect cleanly in a single connected component, then given suitable Hamiltonian perturbations and almost-complex structures, the Floer and Morse (co)chains are isomorphic. As a corollary, cf.~ \cite{Sei99} and \cite[\sc\S 6]{Auy}, if two Lagrangians intersect in several clean components, there is a spectral sequence from the singular cohomologies thereof to the Lagrangian intersection Floer cohomology. Blakey \cite{Blakey} has already extended this result to the Floer homotopy type in the exact case, by more generally allowing intersections locally modelled after the graph of an exact 1-form.
  
  It is this spectral sequence that we seek to upgrade to a spectral sequence living over the Steenrod algebra (or valued in generalized cohomology), in the monotone case. We name the spectral sequence in honor of Oh and Pozniak, whose ideas and technical work constitute the bulk of the proof, even though the result was formulated later by other mathematicians. We begin with a preliminary definition about local systems, which is needed to treat $U$-branes that are not diagonal-respecting. For the purposes of our application in {\sc\S\ref{ssec:proj}}, this may safely be ignored, and replaced with the more widely-known notion of a virtual bundle.

  \begin{DEF}[$R$-local systems]\label{def:ls}
    Given a ring spectrum $R$, and a topological space $X$, an \keywd{$R$-local system} on $X$ is defined formally as a continuous map $\xi : X \to \mathbb Z \times {\rm BGL}_1(R)$, where the latter is the delooping of the $\bb E_1$-group ${\rm GL}_1(R)$ of self-equivalences of $R$ as a module over itself. The space $ \mathbb Z \times {\rm BGL}_1(R)$ can alternatively be described as the full $\infty$-groupoid of all $R$-modules that are abstractly isomorphic to some shift of $R$ (the integer in $\mathbb Z$ records this shift formally). In this language, one can interpret an $R$-local system $\xi$ as a ${\rm Sing}_\bullet X$-shaped diagram in spectra, where each vertex is mapped to an $R$-module that is abstractly isomorphic to some shift of $R$. The \keywd{Thom spectrum} ${\rm Th}(\xi) = X^\xi$ of this local system is simply the colimit of this diagram, taken inside $R$-modules. It should be visualized as the ``total space'' of a ``fibration'' over $X$ with fiber at $x \in X$ given by $\xi(x) \in \bb Z \times {\rm BGL}_1(R)$. The trivial local system of degree $n$ will simply be denoted by $R[n]$, and its Thom spectrum is just $\Sigma^\infty_+ X \wedge R$ shifted by $n$. The Thom isomorphism still holds over any $R$-algebra, viewed as a multiplicative cohomology theory of $R$-modules, given any choice of Thom class (particularly, $H\bb F_2$ is an algebra over $MU$, and every local system is canonically oriented in $H\bb F_2$). Any classical virtual bundle has an induced $R$-local system, via the space-level $J$-homomorphism
    \[ X \to \mathbb Z \times BO \overset{BJ}\to \mathbb Z \times BGL_1(\mathbb S) \to \mathbb Z \times BGL_1(R). \]
    One should visualize the fibers of the associated $\mathbb S$-local system as being the ``one-point compactifications'' of the fibers of the virtual bundle. Smashing with $R$ gives the $R$-local system. A modern reference on local systems and Thom spectra is \cite{ABGHR14}. See also older references \cite{MQRT77, LMSM86} that do not use $\infty$-categories.
  \end{DEF}

  \begin{THM}[Oh-Pozniak spectral sequence]\label{thm:ohpoz}
    Assume $(M,L)$ is given as in the setup {\sc\S\ref{ssec:setup}}, endowed with a $U$-brane, and a cleanly degenerate choice of time-dependent Hamiltonian $H_t$ with $L \cap \Phi_1(L)$ a disjoint union of connected clean intersections.
    
    In a manner similar to the non-degenerate case, each component $C$ gives rise to a $\bb Z$-torsor of capped components $\tilde C$. Let us order all these capped components in ascending order $\{\tilde C_p\}_{p \in \bb Z}$ of their action functional (the order can be made arbitrary when two components have the same action, or one can assemble together all components with the same action into one group.)
    
    Then, for any coefficient system $A$ (or, more generally, $\tau_{\le N_\mu - 3} MU$-module $Z$ with finite homotopy-support), there is a spectral sequence
    \[ E_1^{p,k} = H^k_{\rm sing}(C_p^{\tilde \xi_p}; A) \implies {\rm HF}^k(M,L; A)\]
    of $N_\mu$-periodic $\scr A^*_A(\tau_{\le N_\mu - 3}MU)$-modules, where $\tilde \xi_p$ is an $MU$-local system {\sc(Def.~\ref{def:ls})} on $C_p$, with the property that $\tilde \xi' - \tilde \xi = MU[\delta N_\mu]$ for any two capped components $\tilde C', \tilde C$ lifting the same component $C$, where $\delta$ is the difference element of the two caps in $\pi_2(M,L) = \bb Z$. Here, $p$ is the filtration degree and $k$ is the total degree, and the differential $d_r$ has $(p,k)$-bidegree $(r,1)$. If the $U$-brane is diagonal-respecting, $\tilde \xi_p$ can be chosen to be actual virtual bundles.

    If $H_t = 0$, then the only component of the intersection is $L$ itself, and there is a canonical choice of cap $\tilde L_0$, namely the constant one. If the $U$-brane is diagonal-respecting, then the bundles $\tilde \xi$ can be assumed to be trivial, of rank equal to $\delta N_\mu$, where $\delta$ is the difference element between $\tilde \xi$ and $\tilde \xi_0$ in $\pi_2(M,L) = \bb Z$. So the spectral sequence takes the form
    \[ E_1^{p,k} = H^{k - pN_\mu}(L; A) \implies {\rm HF}^k(M,L;A), \]
    recovering Oh's spectral sequence.
  \end{THM}

  \begin{RMK}[Duality]\label{rmk:dual}
    This spectral sequence comes from a spectrum-level filtration, so one can equally well apply homology instead of cohomology:
    \[ E^1_{p,k} = H_k^{\rm sing}(C_p^{\tilde \xi_p}; A) \Longrightarrow {\rm HF}_k(M, L; A), \]
    with ${\rm deg}_{(p,k)}\,d_r = (-r, -1)$. There is still an action of the Steenrod algebra, because by Spanier-Whitehead duality one has $H_*(X^\vee; A) \cong H^{-*}(X; A)$ if $X^\vee$ denotes the Spanier-Whitehead dual. Now, the underlying spectrum of Floer homology is not finite (in fact, it is only a pro-spectrum), but once again one can consider large enough finite sub-quotients. Hence, both sides have actions of the Steenrod algebra, and one can rewrite at least the $E^1$-page more explicitly as
    \[ E^1_{p,k} = H^{-k}_{\rm sing}(C_p^{-\tilde \xi_p - TC_p}; A), \]
    with the obvious action of the Steenrod algebra. Note that now Steenrod operations decrease the homological degree $k$. We conjecture that this dual spectral sequence should have an intrinsic Floer-theoretic interpretation, coming from reversing the Hamiltonian flow. Indeed, classical Floer cohomology satisfies Poincar\'e duality (which is proven by reversing the flow), and under this identification of Floer homology and cohomology the Oh-Pozniak spectral sequence for the reversed flow is exactly of the same form as the dual Oh-Pozniak spectral sequence for the original flow, that we have just described. What is not a priori clear is whether these two spectral sequences should be canonically isomorphic. Regardless, this technical point is not important at all for our applications in {\sc\S\ref{ssec:proj}}.
  \end{RMK}

  For purely degree-theoretic reasons, one recovers the Albers-PSS isomorphism, over the Steenrod algebra, and likewise for generalized cohomology theories:

  \begin{THM}[Albers-PSS isomorphism]\label{thm:albpss}
    Assume given $(M,L)$ as in the setup {\sc\S\ref{ssec:setup}}, endowed with a diagonal-respecting $U$-brane. Then, there is an isomorphism $H^*_{\rm sing}(L; A) \cong {\rm HF}^*(M,L;A)$ in the range
    \[\numeq\label{eq:pssrange} n - N_\mu + 2 \le * \le N_\mu - 2 , \]
    which moreover is compatible with the $\scr A^*_A(\tau_{\le N_\mu - 3}MU)$-action (note that both the domain and codomain of the given Steenrod operation must lie in this range.) Likewise there is an isomorphism on homology, cf.~the Steenrod algebra action on homology described in {\sc Rmk.~\ref{rmk:dual}} above. If the coefficient system $A$ is replaced by a $\tau_{\le N_\mu - 3} MU$-module $Z$ with finite homotopy-support, then the isomorphism holds for $Z^*$ (or $Z_*$) theory in the range 
    \[n - N_\mu + 2 + d^{\rm max} \le * \le N_\mu - 2 + d^{\rm min},\] where $d^{\rm min}, d^{\rm max}$ are the smallest resp. largest degree of a nonzero $Z^*({\rm pt})$ (or $Z_*({\rm pt})$).
  \end{THM}

  \begin{proof}[Proof using {\sc Thm.~\ref{thm:ohpoz}}]
    The spectral sequence has $E_1$-page consisting of $N_\mu$-shifted copies of $H^*_{\rm sing}(L; A)$. Hence, there are potential differentials between the various copies, but in the total degree range \ref{eq:pssrange} they cannot occur, as the differential has total degree 1. Hence the $E_\infty$-page has exactly one nonzero homology group in this range, which is precisely $H^k_{\rm sing}(L; A)$. So this must coincide with ${\rm HF}^k(M,L;A)$. The statement on homology follows by looking at the dual spectral sequence. The range changes for generalized cohomology theories $Z$, because we must account for the different homological support of $Z^*(L)$, adjusted by $d^{\rm min}$ and $d^{\rm max}$ at the two ends.
  \end{proof}

  The main technical result we will be using for the proof of {\sc Thm.~\ref{thm:ohpoz}} is Pozniak's theorem \cite[{\sc Thm.\,3.4.11}]{Poz} below, which ensures that given any clean intersection of two Lagrangians, one can perturb one of them by a $C^1$-small Hamiltonian deformation, so that the Morse and Floer complexes are isomorphic:

  \begin{PROP}[Pozniak's Isolating Neighborhood]\label{prop:poz} 
    There exists, for each $p$, an open neighborhood $U_p$ of the connected component $C_p$ inside $M$, such that solutions to the Floer equation for a particular $(H_0, J_0)$ carefully chosen, always lie inside $U_p$, and correspond precisely to Morse solutions on $C_p$ for a suitable Morse function and metric $(f_0, g_0)$.
  \end{PROP}

  \begin{proof}
    See the original result \cite[3.4.11]{Poz} of Pozniak. Although he states the result only for cochain complexes, in the proof he makes the stronger claim that solutions to Floer and Morse equations agree. See also \cite[\sc \S5, \S6]{Auy} for a newer exposition, or \cite{Blakey}'s proof of his main application in {\sc \S7.3} in Floer homotopy theory over the sphere spectrum.
  \end{proof}

  \begin{RMK}\label{rmk:hidim}
    This means that when restricting to a Pozniak neighborhood, even those moduli spaces of Maslov difference greater or equal to $N_\mu$ are guaranteed to be well-defined without bubbles, since they agree with standard Morse moduli spaces. It is important to note that the $MU$-orientation extends even over these extra moduli spaces that were a priori not well-defined. We emphasize this because in the proof of {\sc Thm.~\ref{thm:ohpoz}} below we will compare the current Floer construction with the classical Morse construction from our earlier paper \cite{me}, which does use all the higher-dimensional moduli spaces as input.
  \end{RMK}

  \begin{proof}[Proof of \sc Thm.~\ref{thm:ohpoz}]
    Deform the cleanly degenerate Hamiltonian by a $C^2$-small perturbation, so that it becomes non-degenerate and the conclusion of Pozniak's result holds. The resulting capped chords in each neighborhood $\tilde U_p$ can be made to have action arbitrarily close to the original $\tilde C_p$, and hence one can ensure that in the flow category all the capped chords in one $\tilde U_p$ lie in a subsegment, with respect to the total ordering. The Cohen-Jones-Segal construction can be filtered according to this filtration on the capped chords, inducing a spectral sequence on the homology. The Steenrod algebra action is respected by this spectral sequence, since the filtration respects the action of the ring spectrum.
    
    By {\sc Prop.~\ref{prop:poz}}, the graded quotients have Floer moduli spaces (at least as stratified topological manifolds with corners) the same as the Morse moduli spaces. The gradings need not agree (for an analysis of how they differ, see \cite[\sc\S 6.1]{Auy},) and unfortunately neither do the smooth structures. Indeed, recall that Large's construction depends on a sufficiently tame gluing parameter, which unfortunately for the smooth structure we have constructed in \cite[\sc\S3]{me} seems to be the logarithmic one, which is not tame enough for Large's setup.
    
    Thankfully, the Cohen-Jones-Segal construction can be carried out in the topological category (with locally flat embeddings), and the arguments used in sections {\sc\S4,\S5} of our earlier work \cite{me} do not really require a smooth structure, but rather only the abstract existence of moduli spaces $\scr M_{i,j}$ of broken flow lines, and crucially of $\scr W_i$ of broken flow lines with open end (cf.~ {\sc Def.\,3.2} and {\sc Thm.\,3.7} of \cite{me}), which should be contractible.

    It remains to understand what the Cohen-Jones-Segal construction gives for these Morse flow categories. In \cite[\sc Main Thm. iii]{me}, we showed that any framing on the Morse flow category of a closed manifold $C$ recovers the Thom spectrum $C^\xi$ of some virtual bundle $\xi$. In our case, the flow category a priori does not have framings, but merely $U$-structures. Still, the techniques used in \cite{me} work equally well for proving the analogous result involving $MU$-local systems. That is, there is an $MU$-local system $\tilde \xi$ on $C$ that induces the given $MU$-orientations of the Morse flow category (transported from the Floer setup). Note that what we said in {\sc Rmk.~\ref{rmk:hidim}} is needed here, because the higher-dimensional moduli spaces are actually needed in the construction of this local system $\tilde \xi$, cf.~ \cite{me}. The same idea of proof used in \cite{me} adapts here with little modification to show that the Cohen-Jones-Segal construction recovers the Thom spectrum $C^\xi$ over $MU$, and hence also over the weaker $\tau_{\le N_\mu-3} MU$. 
    
    If the $U$-brane is diagonal-respecting, we need to show that these $MU$-local systems $\tilde \xi$ come from actual virtual bundles. Indeed, we note that the embedding 
    $(U_p \cap L) \times^h_{U_p} (U_p \cap L) \hookrightarrow L \times_M L$
    is, up to homotopy, equivalent to the inclusion of $C \hookrightarrow L \overset\Delta\hookrightarrow L \times_M L$. Since the $U$-brane is diagonal respecting, it means that the null-homotopy of the map $L \times^h_M L \to U/O \to BO$ is itself the zero nullhomotopy when restricted to this neighborhood. So all the moduli spaces are actually already stably framed, and the $MU$-orientation is just the induced one. Now we may invoke our original result \cite[\sc Main Thm.-iii]{me} in the setup of stable framings, to conclude that $\tilde\xi_p$ comes from an actual virtual bundle.

    The difference $\tilde \xi' - \tilde \xi$ must be given by the $\delta N_\mu$-shift of $MU$, because the Maslov grading differs precisely by $\delta N_\mu$, while the framings are the same. Finally, if the Hamiltonian is itself zero, and the $U$-brane is diagonal respecting, it follows that framing is the same as the standard one on $L$ in its cotangent bundle (via the Weinstein tubular neighborhood), which recovers the standard suspension spectrum of $L$, without the need to take any Thom space; i.e.~ $\tilde \xi$ can be taken to be zero. For a reference regarding the aforementioned cotangent bundle statement, see Blakey's recent work \cite[\sc Prop. 6.8]{Blakey}.
  \end{proof}

  \begin{RMK}
    Our result can likely be generalized to the so-called quasi-minimally degenerate intersections of Auyeung, cf.~ \cite[\sc Thm.\,6.2]{Auy}. We do not pursue this further here, and content ourselves with clean intersections.
  \end{RMK}

  \subsection{Sample application: Projective spaces}\label{ssec:proj} We now illustrate how this Steenrod algebra action can provide extra information about the possible intersections of $L$ with a Hamiltonian pushoff, beyond what ordinary Floer homology can, in the case of $\bb{RP}^n \subset \bb{CP}^n$.

  \begin{PROP}
    If $n \ge 3$ is odd, then the inclusion $\bb{RP}^n \subset \bb{CP}^n$ satisfies the conditions required in {\sc\S\ref{ssec:setup}}, and admits a diagonal-respecting $U$-brane {\sc(Def.~\ref{def:ubr})}.
  \end{PROP}

  \begin{proof}
    By {\sc Ex.~\ref{exp:sourceu}}, it suffices to show that $T\bb{RP}^n$ extends over $\bb{CP}^n$ as a virtual bundle. By the Euler sequence, one has $[T\bb{RP}^n] = (n+1) \cdot [\scr O_\bb R(1)]$ in $\widetilde{KO}(\bb{RP}^n)$. Since $\scr O_\bb C(1)$ on $\bb{CP}^n$ is the complexification of $\scr O_\bb R(1)$, it becomes isomorphic to $2 \cdot [\scr O_\bb R(1)]$ under the restriction $\widetilde{KO}(\bb{CP}^n) \to \widetilde{KO}(\bb{RP}^n)$. Since $n$ is odd, it follows that $[T\bb{RP}^n]$ is in the image of this restriction map, as desired.
  \end{proof}

  \begin{RMK}
    For even $n$, one can rigorously show that there is no $U$-brane on $\bb{RP}^n \subset \bb{CP}^n$, since the Maslov number $N_\mu = n + 1$ must be even, cf.~ {\sc Rmk.~\ref{rmk:constrbr}}.
  \end{RMK}

  \begin{COR}
    The Floer cohomology ${\rm HF}^*(\bb{CP}^n, \bb{RP}^n; \bb F_2)$ has a well-defined action of $\scr A^*_{\bb F_2}(\tau_{\le n - 3} MU)$, for $n \ge 3$.
  \end{COR}

  \begin{proof}
    By {\sc Thm.~\ref{thm:welldef}}, we get an action of the Steenrod algebra on $\tau_{\le N_\mu - 3} MU = \tau_{\le n - 2} MU$. But $n$ is odd, and $MU$ does not have any odd homotopy groups, so we practically only get a $\tau_{\le n - 3} MU$-structure.
  \end{proof}

  This becomes interesting as soon as $n \ge 5$, since $\tau_{\le 2} MU$ already has interesting cohomology operations, beyond the merely homological Bockstein homomorphism. Indeed, recall from {\sc Prop.~\ref{prop:qimu}} that we have well-defined mod-2 operations $Q_i$ of degree $2^{i+1}-1$, for all $i$ with $2^{i+1} \le n - 1$, satisfying the following rules for modules induced from an actual space $X$:
  \begin{enumerate}
    \item Power rule: $Q_i x = x^{2^{i+1}}$ for all $x \in H^1(X; \bb F_2)$.
    \item Leibniz rule: $Q_i (x \cdot y) = Q_i x \cdot y + x \cdot Q_i y$.
  \end{enumerate}
  Using these properties, the $Q_i$ operations on $\bb{RP}^n$ are very easy to describe:

  \begin{PROP}
    If $x$ denotes the generator of $H^1(\bb{RP}^n; \bb F_2)$, then
    \[ \numeq\label{eq:alphrpn} Q_i(x^k) = k \cdot x^{k+2^{i+1}-1} = \begin{cases}
      x^{k+2^{i+1}-1} & \text{if $k$ odd} \\
      0 & \text{otherwise.}
    \end{cases} \]
  \end{PROP}

  \begin{COR}[Weak Steenrod action computation]
    By the Albers-PSS isomorphism {\sc Thm.~\ref{thm:albpss}}, ${\rm HF}^d(\bb{CP}^n, \bb{RP}^n; \bb F_2) \cong H^d_{\rm sing}(\bb{RP}^n; \bb F_2) = \bb F_2$ for $1 \le d \le n-1$. For degrees $1 \le d \le n - 2^{i+1}$, we have that the operation $Q_i : {\rm HF}^d \to {\rm HF}^{d+2^{i+1}-1}$ lies in the PSS range \ref{eq:pssrange}, and hence is determined by the formula \ref{eq:alphrpn}.
  \end{COR}

  In fact, with more work, the Floer cohomology of $\bb {RP}^n$ can be computed even in the degrees $0$ and $n$ modulo $N_\mu = n+1$, which the Albers-PSS map does not see:

  \begin{THM}[due to Oh \cite{Oh93}]
     ${\rm HF}^d(\bb{CP}^n, \bb{RP}^n) = \bb F_2\<x_d\>$ for all $d \in \bb Z/(n+1)$.
  \end{THM}

  This forces the Oh-Pozniak spectral sequence to collapse, so we get the stronger

  \begin{THM}[Strong Steenrod action computation]\label{thm:strong}
    For $0 \le d \le n-2^{i+1}+1$, we have that
    \[ Q_i (x_d) = d \cdot x_{d+2^{i+1}-1} = \begin{cases}
      x_{d+2^{i+1}-1} & \text{if $d$ odd} \\
      0 & \text{otherwise.}
    \end{cases} \]
  \end{THM}

  \begin{RMK}\label{rmk:transgress}
    We still do not know what $Q_i$ look like when they transgress the boundary between the various $(n+1)$-shifted copies of $H^*_{\rm sing}(\bb{RP}^n; \bb F_2)$, i.e.~ in degrees
    \[n - 2^{i+1} + 2 \le d \le n.\]
    This is because the spectral sequence only gives us the associated graded. We suspect that the extra information is hidden in the real Gromov-Witten theory of $\bb{RP}^n \subset \bb{CP}^n$, but do not pursue this question further in this paper.
  \end{RMK}

  Now, we turn to an application. A natural question to ask, in the spirit of the Arnol'd conjecture, is what kinds of manifolds can arise as clean intersections between $\bb{RP}^n$ and a Hamiltonian perturbation thereof, in $\bb{CP}^n$. We only consider the two ``easiest cases,'' namely when the intersection has a single connected component, and then when it is the disjoint union of a single point and some other connected component. We begin by gleaning consequences that are purely homological, and end this section with two non-trivial applications of the Steenrod action.

  \begin{THM}[Connected clean intersection]\label{thm:conn}
    The only connected clean intersection between $\bb{RP}^n \subset \bb{CP}^n$ and a Hamiltonian isotopy thereof is the whole of $\bb{RP}^n$.
  \end{THM}

  \begin{proof}[Proof, using only Floer cohomology]
    Assume $C = L \cap \Phi_1(L)$ has dimension $d < n$. The Oh-Pozniak spectral sequence must have $E_\infty$ page given by a single copy of $\bb F_2$ in each total degree. But the $E_1$ page of the spectral sequence already must have at least one total degree where it is entirely zero, since the homological amplitude of $C$ is $d + 1 < N_\mu$. This forces $C$ to have dimension $n$, i.e.~ $C = \bb{RP}^n$.
  \end{proof}

  \begin{THM}[Point + Connected, part I]\label{thm:pt+connI}
    If $\{p\} \sqcup C = \bb{RP}^n \cap \Phi_1(\bb{RP}^n)$ is a clean intersection between $\bb{RP}^n$ and a Hamiltonian isotopy $\Phi_1(\bb{RP}^n)$ inside $\bb{CP}^n$, with $C$ connected, then $C$ must be $(n - 1)$-dimensional, and its cohomology with $\bb F_2$-coefficients must be the same as that of $\bb{RP}^{n-1}$, i.e.~ $H^i(C; \bb F_2) = \bb F_2\<y_i\>$ for $0 \le i \le n-1$.
  \end{THM}

  \begin{RMK}\label{rmk:exofc}
    Before proving this, we note that the case $C = \bb{RP}^{n-1}$ actually does occur in reality, so the statement above is not vacuous. Indeed, the standard Morse-Bott function
    \[ f([x_0 : \cdots : x_n]) = \frac{\pm |x_0|^2}{|x_0|^2 + \cdots+ |x_n|^2} \]
    can be used to push off $\bb{RP}^n$ in a Weinstein neighborhood of it, and get precisely two connected clean components, one being a point $\{p\} := \{x_1 = \cdots = x_n = 0\}$, and the other being $C := \bb{RP}^{n-1}$, given by $\{x_0 = 0\}$. One has $\mu(p) \in \{0,n\}$, depending on the sign. We would conjecture that $C = \bb{RP}^{n-1}$ is the only such example, (at least up to homotopy-equivalence), and that $\mu(p) \in \{0,n\}$ always, but our techniques are only powerful enough to prove certain results about the cohomology thereof, and its Steenrod algebra action.
  \end{RMK}

  \begin{proof}[Proof, using only Floer cohomology]
    Once again, the $E_\infty$ page must be nontrivial in each total degree, and hence so does the $E_1$ page. Since $C$ must have dimension at most $n - 1$, the only way this can happen is if the cohomological supports of $C$ and the point $p$ completely tile $\bb Z/(n+1)$. In particular, there must be no differential in the spectral sequence, since this would kill the class coming from the point $p$, leading to the same contradiction as before. So the spectral sequence collapses, and $C$ must have dimension $n - 1$ and cohomology $\bb F_2$ in each degree. Cf.~ also {\sc Fig.~\ref{fig:sss}}.
  \end{proof}

  \begin{figure}[h]
    \centering
    \includegraphics[scale=1.5]{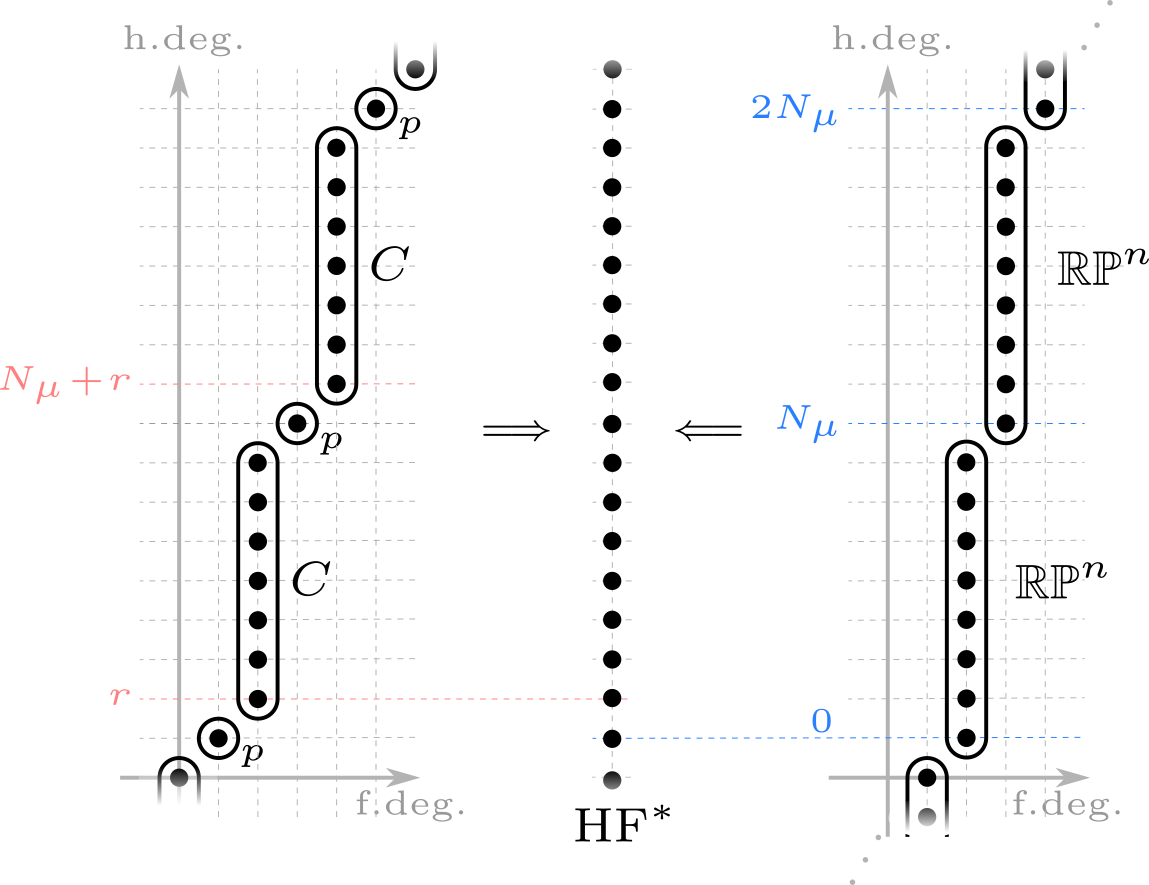}
    \caption{Depiction of the spectral sequence for clean intersection $\{p\} \sqcup C$ (left) and for the whole of $\bb{RP}^n$ (right), in the sample case $n = 7$ and $r = 1$, following notation conventions of part II of the theorem.}\label{fig:sss}
  \end{figure}

  \begin{THM}[Point + Connected, part II] \label{thm:pt+connII}
    Assume $n \ge 3$ odd. Further, denote by $r \in \{0, \ldots, n\}$ the unique residue such that $\mu(p) \equiv r - 1 \; ({\rm mod\;} n+1)$, and by $y_d$ the generator of $H^d(C; \bb F_2)$ for $0 \le d \le n-1$. Then, for all $0 \le d \le n - \max(1,r) - 2^{i+1} + 1$, one has
    \begin{equation}\label{eq:idq}
      Q_i (y_d) = (r + d) \cdot y_{d+2^{i+1}-1} + r \cdot y_{2^{i+1}-1}y_d.
    \end{equation}
    Dually, for all $0 \le d \le \min\big\{n-1, \left[(r - 2)\;{\rm mod}\;(n+1)\right]\big\} - 2^{i+1} + 1$, one has
    \begin{equation}\label{eq:idqdual}
      Q_i (y_d) = (r + d + 1) \cdot y_{d+2^{i+1}-1} + (r + 1) \cdot y_{2^{i+1}-1} y_d,
    \end{equation}
    where $[(r - 2)\;{\rm mod}\;(n+1)]$ is the remainder of $r - 2$ modulo $n + 1$. Refer to {\sc Fig.~\ref{fig:ranges}} for a visualization of the two ranges, as functions of $r$.
  \end{THM}

  \begin{figure}[h]
    \centering
    \includegraphics{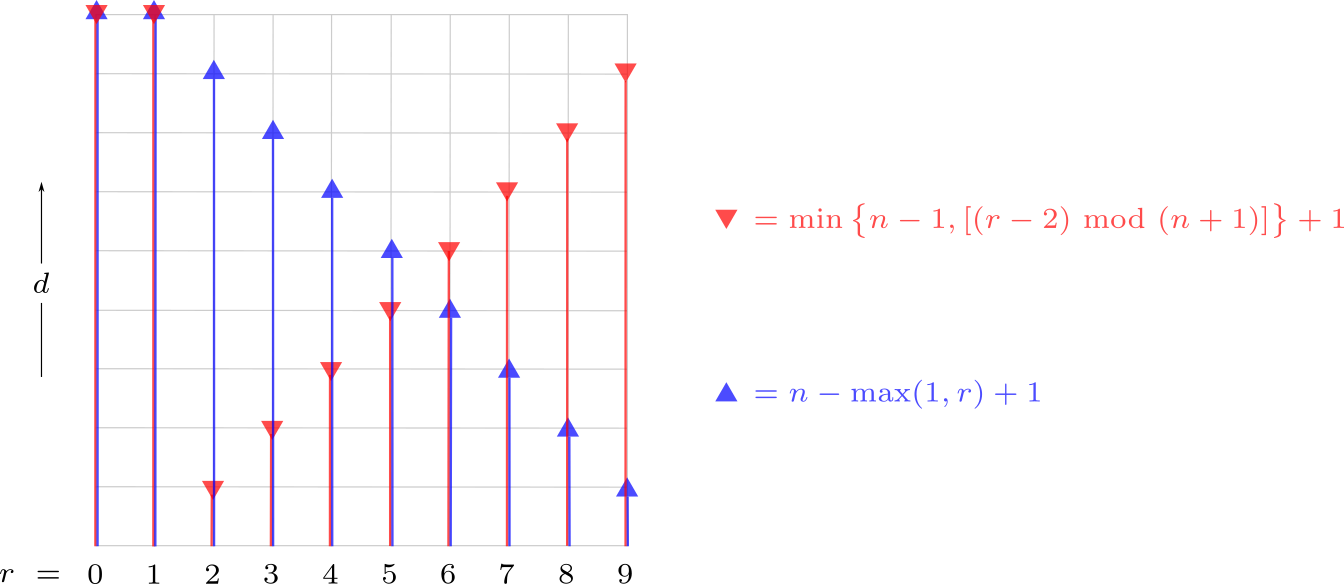}
    \caption{Depiction of the two ranges in which $d + 2^{i+1}$ must lie, for the identities (\ref{eq:idq}), (\ref{eq:idqdual}) to hold, as a function of $r$, in the sample case of $n = 9$.}\label{fig:ranges}
  \end{figure}

  \begin{proof}[Proof, using the Steenrod action]
    First, let us explain how the second identity follows from the first by ``duality,'' cf.~ {\sc Rmk.~\ref{rmk:dual}}. It is easier to interpret duality as coming from reversing the Hamiltonian: 
    \[ {\rm HF}^*(M,L;H_t) \cong {\rm HF}^*(M, \Phi_1(L); -H_{-t}). \]
    As Albers \cite{Albers} explains in the paragraph before the one containing equation (3.2), reversing the Hamiltonian induces a change in Maslov grading by $\mu^{-H_{-t}}(\tilde x) = n - \mu^{H_t}(\tilde x)$. In this case the old $r$ is replaced by $(n - r + 2) \; {\rm mod}\; n + 1$. The effect of this mod 2 is changing the parity of $r$, hence the $+1$'s in the parentheses in the second equation.
    
    So it remains to prove the first identity. First, notice that $r$ must be the rank of $\tilde \xi$ mod $(n + 1)$ (following the notation of {\sc Thm.~\ref{thm:ohpoz}}), because this is the only way to ensure that the spectral sequence does not have any total degree in which the $E_\infty$-page is zero. Now, the point is that in the range $0 \le d \le n - \max(1,r)$, we have a degree-$r$ isomorphism induced by the Oh-Pozniak spectral sequences associated to both filtrations:
    \[ H^{d+r}_{\rm sing}(\bb{RP}^n; \bb F_2) \overset\sim\longleftrightarrow {\rm HF}^{d+r}(M,L; \bb F_2) \overset\sim\longleftrightarrow H^{d+r}_{\rm sing}(C^{\tilde \xi}; \bb F_2),\] taking $x_{r+d} \leftrightarrow \theta y_d$ (where $\theta$ is the Thom class of $\tilde \xi$, and the $x$'s and $y$'s follow the notations of {\sc Thms.~\ref{thm:strong},~\ref{thm:pt+connI}}.) Cf.~ again {\sc Fig.~\ref{fig:sss}}. So for $0 \le d \le n - \max(1,r) - 2^{i+1} + 1$, both the domain and codomain of the $Q_i$ operation lie in the aforementioned range, and in particular under this correspondence one has
    \[\numeq\label{eq:corr} (r+d)x_{r+d+2^{i+1}-1} = Q_i(x_{r+d}) \leftrightarrow Q_i(\theta y_d) = Q_i\theta \cdot y_d + \theta \cdot Q_i y_d. \]
    Note that in particular, setting $d = 0$, we have $r x_{r+2^{i+1}-1} \leftrightarrow Q_i\theta \cdot y_0 = Q_i\theta$. Since the correspondence is bijective this means that $Q_i \theta = \theta \cdot (r y_{2^{i+1}-1})$, and more generally for arbitrary $d$ we get
    \[ (r+d)y_{d+2^{i+1}-1} = r \cdot  y_{2^{i+1}-1} y_d + Q_i y_d, \]
    after dividing by the Thom class. This concludes the proof, after rearranging factors.
  \end{proof}

  Here are some nicer-looking consequences:

  \begin{COR}\label{cor:pt+conn}
    One has
    \begin{enumerate}
      \item If $d$ lies in both ranges, then 
      \[y_{d+2^{i+1}-1} = y_d y_{2^{i+1}-1}, \text{ and } Q_i y_{d} = dy_{d+2^{i+1}-1}.\]
      The values of $r$ for which the conclusion is the strongest, cf.~ {\sc Fig.~\ref{fig:ranges}}, are $r \in \{0,1\}$, i.e.~ exactly those emerging from the examples of {\sc Rmk.~\ref{rmk:exofc}}, for which the intersection of the two ranges is $0 \le d \le n - 2^{i+1}$.
      \item When $d \le \frac{n+1}2 - 2^{i+1}$, then $d$ lies in at least one of the two ranges. In particular, $Q_i(y_d)$ satisfies some formula expressing it as a linear combination of $y_{d + 2^{i+1} - 1}$ and $y_{2^{i+1}-1} y_d$.
      \item Further specializing to $d = 1$, for $2^{i+2} < n$, we get a formula for $y_1^{2^{i+1}}$.
    \end{enumerate}
    Note that for the first identity in part {\sc i}, and for part {\sc iii}, no knowledge of cohomology operations is required to understand the statement.
  \end{COR}

  \begin{proof}
    The first fact follows by adding both identities mod 2, and then substituting back into either of them. The second one just follows from elementary arithmetic. The third follows from the power rule $Q_iy_1 = y_1^{2^{i+1}}$.
  \end{proof}

  We end with an application that allows us to compute some characteristic classes of this connected clean component $C$.

  \begin{DEF}\label{def:qi}
    Given a virtual bundle $E$ on a space $X$, we define $q_i(E) \in H^{2^{i+1}-1}(X; \bb F_2)$ to be the characteristic class
    \[ q_i(E) := \theta_E^{-1} \cdot Q_i(\theta_E), \]
    where $\theta_E$ is the Thom class of $X^E$, and $\theta_E^{-1}\cdot$ denotes the inverse of the Thom isomorphism.
  \end{DEF}

  \begin{PROP}
    One has
    \begin{enumerate}
      \item $q_i(E \oplus F) = q_i(E) + q_i(F)$ for any two bundles $E, F$;
      \item $q_i(L) = w_1(L)^{2^{i+1}-1}$ for any line bundle $L$;
      \item In terms of the Stiefel-Whitney roots $\{x_j\}$, $q_i$ is given by the symmetric polynomial $\sum_j x_j^{2^{i+1}-1}$, which can be used to express $q_i$ in terms of the $w_k$'s, e.g.
      \[ q_0 = w_1, \qquad q_1 = w_1^3 + w_1w_2 + w_3. \]
    \end{enumerate}
  \end{PROP}

  \begin{proof}
    For part {\sc i}, note that $E \oplus F$ is the restriction of $E \boxplus F$ via the diagonal, and that $\theta_{E \boxplus F} = \theta_E \otimes \theta_F$ under the isomorphism ${\rm Th}(E \boxplus F) \cong {\rm Th}(E) \wedge {\rm Th}(F)$. Therefore, by the Leibniz rule we have
    \[ \begin{aligned}
      q_i(E \boxplus F) &= (\theta_E \otimes \theta_F)^{-1} Q_i(\theta_E \otimes \theta_F) = (\theta_E \otimes \theta_F)^{-1} (Q_i\theta_E \otimes \theta_F) + (\theta_E \otimes \theta_F)^{-1} (\theta_E \otimes Q_i\theta_F) \\
      &= \theta_E^{-1} Q_i(\theta_E) \otimes 1 + 1 \otimes \theta_F^{-1} Q_i(\theta_F) = q_i(E) \otimes 1 + 1 \otimes q_i(F).
    \end{aligned} \]
    Pulling back via the diagonal gives the desired $q_i(E \oplus F) = q_i(E) + q_i(F)$.

    For part {\sc ii}, it suffices to prove the formula for the universal case of $\scr O(1)$ on $\bb{RP}^\infty$. Since the Thom space of $\scr O(1)$ is well-known to also be homeomorphic to $\bb{RP}^\infty$, it follows that the Thom class $\theta_{\scr O(1)}$ generates the (nonunital) cohomology ring of the Thom space. We already know that on rank-1 classes, the $Q_i$ satisfies the power rule, i.e.~
    \[ Q_i(\theta_{\scr O(1)}) = \theta_{\scr O(1)}^{2^{i+1}}. \]
    Thus,
    \[ q_i(\scr O(1)) = \theta_{\scr O(1)}^{-1} \theta_{\scr O(1)}^{2^{i+1}} = [\scr O(1)]^{2^{i+1}-1}, \]
    as desired.

    Part {\sc iii} now follows by the Splitting Principle from parts {\sc i} and {\sc ii}.
  \end{proof}

  \begin{COR}\label{cor:qirpn}
    On $\bb{RP}^n$, one has
    \[ q_i(T\bb{RP}^n) = q_i(-T\bb{RP}^n) = \begin{cases}
      x^{2^{i+1}-1} & \text{if } n \text{ even}, \\
      0 & \text{if } n \text{ odd},
    \end{cases} \]
    where $x \in H^1(\bb{RP}^n; \bb F_2)$ is the generator.
  \end{COR}

  \begin{proof}
    By the Euler sequence, one has $T\bb{RP}^n \oplus \bb R \cong \scr O(1)^{\oplus (n+1)}$, and the result now follows from the properties we have just proven.
  \end{proof}

  \begin{THM}[Point + Connected, part III]\label{thm:pt+connIII}
    With the same conventions as in part II, one has the non-vanishing result
    \[ q_i(TC) = y_{2^{i+1}-1} \]
    for all $i \ge 0$ satisfying $2^{i+1} \le \min\big\{n-1, n-r, \left[(r - 2)\;{\rm mod}\;(n+1)\right]\big\} + 1$, cf.~again the intersection of the two conditions in {\sc Fig.~\ref{fig:ranges}}.
  \end{THM}

  \begin{proof}
    We begin by noting that $q_i(TC) = q_i(-TC)$ by additivity, so it suffices to compute the latter. Also, we remark that the condition imposed on $i$ is precisely equivalent to the intersection of the two conditions from part {\sc ii} of the theorem, for $d = 0$ (i.e.~ we want to ensure that there exists at least one $d$ satisfying both constraints, and if any $d$ works, then certainly $d = 0$ works.)

    Referring back to {\sc Thm.~\ref{thm:pt+connII}}, we have proved that for all $0 \le d \le n - \max(1, r) - 2^{i+1} + 1$, one has, cf.~ \ref{eq:corr}
    \[ (r + d) x_{r+d+2^{i+1}-1} = Q_i(x_{r+d}) \leftrightarrow Q_i(\theta_{\tilde \xi} y_d) = \theta_{\tilde \xi} \cdot (q_i(\tilde \xi) y_d + Q_iy_d), \]
    where now we emphasize the dependence of the Thom class $\theta_{\tilde\xi}$ on $\tilde \xi$ by its subscript. Setting $d = 0$, we get
    \[ r x_{r + 2^{i+1} - 1} = Q_i(x_r) \leftrightarrow Q_i(\theta_{\tilde \xi}) = \theta_{\tilde \xi} \cdot q_i(\tilde \xi), \]
    since $y_0$ is the identity in the cohomology ring of $C$. Since $x_{r+2^{i+1}-1} \leftrightarrow \theta_{\tilde \xi}y_{2^{i+1}-1}$ under the correspondence, it follows (after dividing by the Thom class) that
    \[ \numeq\label{eq:c1} r y_{2^{i+1} - 1} = q_i(\tilde \xi). \]
    Analogously, using the dual spectral sequence, cf.~ {\sc Rmk.~\ref{rmk:dual}}, we have
    \[ r' x_{r'+2^{i+1}-1} = Q_i(\theta_{-T\bb{RP}^n} x_{r'}) \leftrightarrow \theta_{-\tilde \xi - TC} \cdot q_i(-\tilde \xi - TC), \]
    where $r' := (n - r + 2)$ mod $n + 1$. We have implicitly used that $q_i(-T\bb{RP}^n) = 0$, cf.~ {\sc Cor.~\ref{cor:qirpn}}. Again, since $x_{r'+2^{i+1}-1} \leftrightarrow \theta_{-\tilde \xi - TC} y_{2^{i+1}-1}$, we get
    \[ \numeq\label{eq:c2} r' y_{2^{i+1} - 1} = q_i(-\tilde \xi - TC). \]
    Finally, keeping in might that modulo 2 one has $r' \equiv r + 1$, and adding the two equations \ref{eq:c1}, \ref{eq:c2}, we get
    \[ y_{2^{i+1} - 1} = q_i(-TC), \]
    as desired.
  \end{proof}

  \begin{RMK}[Further directions]\label{rmk:next}
    The applications in this final subsection are meant to illustrate the techniques that emerge from our general theory, and are by no means sharp, or exhaustive. For example, we expect the following ideas to also be fruitful:
    \begin{enumerate}
      \item More generally, if $\bb{RP}^n \cap \Phi_1(\bb{RP}^n)$ is a disjoint union of $k > 0$ connected components $C_1 \sqcup \cdots \sqcup C_k$, which moreover satisfy $\dim C_1 + \cdots + \dim C_k \le n - k + 1$, then the inequality is in fact sharp, they must have the same $\bb F_2$-cohomology as the real projective spaces of the respective dimensions, and likewise similar results should hold at the level of $Q_i$ Steenrod operations and $q_i$ characteristic classes.
      \item If $\bb{RP}^n \subset \bb{CP}^n$ is replaced by the real locus of a real complete intersection $L \subset M$, at least when $\pi_1 L \to \pi_{1}\bb{RP}^n = \mathbb Z/2$ is an isomorphism, the pair $(M,L)$ should once again satisfy the same constraints as in the beginning of {\sc\S\ref{ssec:setup}}, and we once again expect the same techniques to produce analogous results, which should be particularly interesting when the Maslov number is roughly smaller than the dimension of $L$, but greater than half of it, cf.~ the Albers-PSS range \ref{eq:pssrange}.
      \item Analogous results should work for generalized cohomologies $Z^*$ with finite homotopy-support, where the ends of the ranges are shaved off appropriately, to account for the fact that $Z^*$ is not concentrated in degree 0. An interesting place to start would be a Postnikov truncation of the connective Morava $k(n)$-theory, since the Atiyah-Hirzebruch spectral sequence has $d_2$ differential exactly the $Q_n$ operation. Roughly, we learn that $k(n)$-theory of $C^\xi$ must be zero in a certain range.
      \item There does not seem to be anything preventing one from considering intersection Floer cohomology of two distinct (i.e.~ not Hamiltonian isotopic Lagrangians) $L, K \subset M$, or in relaxing the topological conditions $\pi_1(M,L) = 0, \pi_2(M,L) = \bb Z$. The trouble, however, is that the theory will be ${\rm gcd}(N_\mu(L), N_\mu(K))$-periodic, so the range will significantly decrease unless we choose Lagrangians with appropriate Maslov numbers. Also, if $\pi_2(M,L)$ is too big, one will begin to encounter infinitely many critical points with the same Maslov index, in which case working over some suitable notion of Novikov ring becomes necessary.
      \item As alluded to in {\sc Rmk.\ref{rmk:transgress}}, it should morally not be that difficult to complete the calculation of the Steenrod squares, provided one interprets them correctly in terms of real Gromov-Witten theory. This would likely sharpen the ranges in which we can say something about $Q_i$ and $q_i$. Cf.~ also {\sc Rmk.~\ref{rmk:obs}} about real Gromov-Witten theory potentially being related to obstructions seen at the level of Steenrod algebra actions on Floer cohomology.
      \item Note that the non-vanishing of $Q_i$ implies the non-vanishing of all the ${\rm Sq}^{2^j}$, for all $j \le i$. Thus, it is likely that Blakey \cite{Blakey}'s ideas could be adapted to the monotone setup, potentially providing new lower bounds for the number of (possibly degenerate) intersection points with a Hamiltonian isotopy.
    \end{enumerate} 
  \end{RMK}

    \newpage
    \bibliographystyle{abbrvurl} 
    \bibliography{MonLagr}{} 

\end{document}